\documentclass[letterpaper]{amsart}
\usepackage[utf8]{inputenc}
\usepackage[T1]{fontenc}
\usepackage{MHmacros}
\usepackage{enumitem}

\newtheorem{theorem}{Theorem}%[section]
\newtheorem{corollary}[theorem]{Corollary}
\newtheorem{lemma}[theorem]{Lemma}
\newtheorem{proposition}[theorem]{Proposition}
\newtheorem{question}[theorem]{Question}

\theoremstyle{definition}
\newtheorem{definition}[theorem]{Definition}

\newcommand{\ILd}{\(I\)-\,\(\Ld_\kappa\)}
\newcommand{\ILdthetasc}{\(I\)-\,\(\Ldthetasc_\kappa\)}

\begin{document}
\title{Joint diamonds and Laver diamonds}
\author{Miha E.\ Habič}
\address{
	Faculty of Information Technology\\
	Czech Technical University in Prague\\
	Th\'akurova 9\\
	160 00 Praha 6\\
	Czech Republic
	\&
	Department of Logic\\
	Faculty of Arts\\
	Charles University\\
	n\'am.\ Jana Palacha 2\\
	116 38 Praha 1\\
	Czech Republic
}
\email{habicm@ff.cuni.cz}
\urladdr{https://mhabic.github.io}
\thanks{The author is grateful to the anonymous referee for their numerous suggestions, which greatly improved the presentation in the paper.\\
The material presented in this paper is part of the author's doctoral dissertation, 
completed at the Graduate Center, CUNY, under the supervision of Joel David Hamkins.\\
The author was partly supported by the ESIF, EU Operational Programme Research, Development and Education, the International Mobility of Researchers in CTU project no.~(CZ.02.2.69/0.0/0.0/16\_027/0008465) at the Czech Technical University in Prague, and the joint FWF--GA\v{C}R grant no.~17-33849L: Filters, Ultrafilters and Connections with Forcing.}

\begin{abstract}
The concept of jointness for guessing principles, specifically \(\diamond_\kappa\)
and various Laver diamonds, is introduced. A family of guessing sequences is joint if the elements of any
given sequence of targets may be simultaneously guessed by the members of the family.
While equivalent in the case of \(\diamond_\kappa\), joint Laver diamonds
are nontrivial new objects. We give equiconsistency results for most of the large
cardinals under consideration and prove sharp separations between joint Laver diamonds
of different lengths in the case of \(\theta\)-supercompact cardinals.
\end{abstract}

\maketitle

\section{Introduction}
The notion of a Laver function, introduced for supercompact cardinals 
in~\cite{Laver1978:MakingSupercompactnessIndestructible},
is a powerful strengthening of the usual \(\diamond\)-principle to the large cardinal setting.
It is based on the observation that a large variety of large cardinal properties give rise to
different notions of a ``large'' set, intermediate between stationary and club, and these
are then used to provide different guessing principles, where we require that the
sequence guesses correctly on these ``large'' sets. This is usually recast in terms of
elementary embeddings or extensions (if the large cardinal in question admits such a
characterization), using various ultrapower constructions. For example, in the case of a 
supercompact cardinal \(\kappa\), the usual definition states that a \emph{Laver function}
for \(\kappa\) is a function \(\ell\colon \kappa\to V_\kappa\) such that for any \(\theta\)
and any \(a\in H_{\theta^+}\) there is a \(\theta\)-supercompactness embedding
\(j\colon V\to M\) with critical point \(\kappa\) such that \(j(\ell)(\kappa)=a\) (this
ostensibly second order definition can be rendered in first order language by replacing
the quantification over arbitrary embeddings with quantification over
ultrapowers by measures on \(\power_\kappa(\theta)\), as in Laver's original account).
In this example, Łoś's theorem tells us that the set of \(\alpha<\kappa\), for which
\(\ell(\alpha)\) codes an ``initial segment'' of \(a\), is large, in the sense that
it has measure 1 with respect to the normal measure on \(\kappa\) derived from \(j\). 

Laver functions for other large cardinals were later defined by Gitik and 
Shelah~\cite{GitikShelah1989:IndestructibilityOfStrong}, Corazza~\cite{Corazza1989:LaverSequencesExtendible}, 
Hamkins~\cite{Hamkins2002:StrongDiamondPrinciples}, and others. The term \emph{Laver diamonds} 
has been suggested to more strongly underline the connection between
the large and small cardinal versions.

In this paper we examine the notion of \emph{jointness} for both ordinary and Laver
diamonds. We shall give a simple example in Section~\ref{sec:defs}; for now let us just
say that a family of Laver diamonds is joint if they can guess their targets simultaneously
and independently of one another.
Section~\ref{sec:defs} also introduces some terminology that will ease later discussion.
Sections~\ref{sec:JLDSc} and~\ref{sec:JLDStrong} deal with the outright existence
or at least the consistency of the existence of joint Laver sequences for supercompact and strong
cardinals, respectively. Our results will show that in almost all cases the existence of
a joint Laver sequence of maximal possible length is simply equiconsistent
with the particular large cardinal. The exception are the \(\theta\)-strong cardinals
where \(\theta\) is a limit of small cofinality, for which we prove that additional
strength is required for even the shortest joint sequences to exist. 
We also show that there are no
nontrivial implications between the existence of joint Laver sequences of different lengths.
Section~\ref{sec:JD} considers joint \(\diamond_\kappa\)-sequences and their relation
to other known principles. Our main result there shows that, for a fixed \(\kappa\), the 
principle \(\diamond_\kappa\) is simply
equivalent to the existence of a joint \(\diamond_\kappa\)-sequence of any possible length.

We shall list open questions wherever they arise in the course of exposition.

\section{Jointness: a motivating example}
\label{sec:defs}
All of the large cardinals we will be dealing with in this paper are characterized by
the existence of elementary embeddings of the universe into inner models 
which have that cardinal
as their critical point. We can thus speak of embeddings associated to a measurable, 
a \(\theta\)-strong, a 17-huge cardinal, and so forth.
At this stage we do not insist that these embeddings are any kind of ultrapower embedding, or
even definable, so this whole introductory discussion should take place in an appropriate
second-order setting.
Since the definitions of (joint) Laver diamonds for these various large cardinals
are quite similar, we give the following general definition as a framework to
aid future exposition.

\begin{definition}
\label{def:guessingFunctions}
Let \(j\) be an elementary embedding of the universe witnessing the largeness of
its critical point \(\kappa\) (a measurable or a \((\kappa+2)\)-strongness embedding, for instance)
and let \(\ell\) be a function defined on \(\kappa\).
We say that a set \(a\), the target, is \emph{guessed by \(\ell\) via \(j\)} 
if \(j(\ell)(\kappa)=a\).

If \(A\) is a set or a definable class, say that \(\ell\) is an \emph{\(A\)-guessing Laver 
function} (or \emph{Laver diamond}) if for any \(a\in A\) there is an embedding \(j\), witnessing the largeness of 
\(\kappa\), such 
that \(\ell\) guesses \(a\) via \(j\). If there is an \(A\)-guessing Laver function
for \(\kappa\), we shall say that \(\Ld_\kappa(A)\) holds.\footnote{Different notation has been used by different authors to denote the existence of a Laver function. We chose here to follow Hamkins~\cite{Hamkins2002:StrongDiamondPrinciples}.}
\end{definition}

To simplify the terminology even more, we shall associate to each type of large cardinal 
considered, a default set of targets \(A\) (for example, when talking about a measurable cardinal \(\kappa\), we will be predominantly interested in targets from \(H_{\kappa^+}\)). 
In view of this, whenever we neglect
the mention of a particular class of targets, these default targets will be intended.

We will often specify the type of large cardinal embeddings we
have in mind explicitly, by writing \(\Ldmeas_\kappa\), or \(\Ldthetasc_\kappa\), or similar.
This is to avoid ambiguity; for example, we could conceivably start with a supercompact
cardinal \(\kappa\) but only be interested in its measurable Laver functions.
Even so, to keep the notation as unburdened as possible, we may sometimes omit
the specific large cardinal property under consideration when it is clear from
context.

As a further complication, the stated definition of an \(A\)-guessing Laver function is 
second-order, since we are quantifying over all possible embeddings \(j\). This is 
unavoidable for arbitrary \(A\). However, the default sets of targets we shall be working
with are chosen in such a way that standard factoring arguments allow us to restrict
our attention to ultrapower embeddings (by measures or extenders). The most relevant definitions of
Laver functions can therefore be recast in first-order language in the usual way.

Given the concept of a Laver diamond for a large cardinal \(\kappa\), we might ask
when two Laver functions are different and how many distinct ones can \(\kappa\)
carry. It is clear that the guessing behaviour of these functions is determined by their restrictions to large (in the sense of an appropriate large-cardinal measure) sets; in other
words, \(j(\ell)(\kappa)\) and \(j(\ell')(\kappa)\) equal one another if \(\ell\) and 
\(\ell'\) only differ on a small (nonstationary, say) subset of their domain.
We definitely do not want to count these functions as distinct: they cannot even guess
distinct targets! Instead, what we want are Laver functions whose targets, under a single
embedding \(j\), can be chosen completely independently. Let us illustrate this
situation with a simple example.

Suppose \(\ell\colon\kappa\to V_\kappa\) is a supercompactness Laver function as defined
in the introduction. We can
then define two functions \(\ell_0,\ell_1\) by letting \(\ell_0(\xi)\) and
\(\ell_1(\xi)\) be the first and second components, respectively, of \(\ell(\xi)\), if
this happens to be an ordered pair. These two are then easily seen to be Laver functions
themselves, but have the additional property that, given any pair of targets \(a_0,a_1\),
there is a \emph{single} supercompactness embedding \(j\) such that
\(j(\ell_0)(\kappa)=a_0\) and \(j(\ell_1)(\kappa)=a_1\) (just the one that makes \(j(\ell)(\kappa)=(a_0,a_1)\)). 
This additional trait, where two Laver functions are, in a sense, enmeshed, we call jointness.

\begin{definition}
\label{def:jointLaverDiamonds}
Let \(A\) be a set or a definable class and let \(\kappa\) be a cardinal with a notion of 
\(A\)-guessing Laver function. A sequence \(\vec{\ell}=\langle \ell_\alpha;\alpha<\lambda\rangle\)
of \(A\)-guessing Laver functions is an \emph{\(A\)-guessing joint Laver sequence} if for any 
sequence \(\vec{a}=\langle a_\alpha;\alpha<\lambda\rangle\) of targets from \(A\) there is a single 
embedding \(j\), witnessing the largeness of \(\kappa\), such that each 
\(\ell_\alpha\) guesses \(a_\alpha\) via \(j\). If there is
an \(A\)-guessing joint Laver sequence of length \(\lambda\) for \(\kappa\), we shall say 
that \(\jLd_{\kappa,\lambda}(A)\) holds.
\end{definition}

In other words, a sequence of Laver diamonds is joint if, given any sequence of targets,
these targets can be guessed simultaneously by their respective Laver diamonds.

We must be careful to distinguish between entire sequence being jointly Laver
and its members being pairwise jointly Laver. It is not difficult to find examples of
three (or four or even infinitely many) 
Laver functions that are pairwise joint but not fully so. For example,
given two joint Laver functions \(\ell_0\) and \(\ell_1\), we might define
\(\ell_2(\xi)\) to be the symmetric difference of \(\ell_0(\xi)\) and \(\ell_1(\xi)\). It is easy
to check that any two of these three functions can have their targets freely chosen, but
the third one is uniquely determined by the other two.

Jointness also makes sense for ordinary diamond sequences, but needs to be formulated
differently, since elementary embeddings do not (obviously) appear in that setting.
Rather, we distil jointness for Laver diamonds into a property of certain
ultrafilters and then apply this to more general filters and diamond sequences.
We explore this further in Section~\ref{sec:JD}.

\section{Joint Laver diamonds for supercompact cardinals}
\label{sec:JLDSc}

\begin{definition}
A function \(\ell\colon\kappa\to V_\kappa\)
is a \(\theta\)-supercompactness Laver function for \(\kappa\) if it guesses elements of 
\(H_{\theta^+}\) via \(\theta\)-supercompactness embeddings with critical point \(\kappa\).
This also includes the case of \(\kappa\) being measurable (as this is equivalent to
it being \(\kappa\)-supercompact).

If \(\kappa\) is fully supercompact, then a function \(\ell\colon\kappa\to V_\kappa\) is
a Laver function for \(\kappa\) if it is a \(\theta\)-supercompactness Laver function
for \(\kappa\) for all \(\theta\).
\end{definition}

%These definitions, as given, are fundamentally of second order, since we are
%quantifying over all elementary embeddings of the universe. We could, of
%course, reformulate them in a first-order way, by requiring the existence of
%certain measures, whose ultrapower embeddings witness the required Laver property.
%Standard factoring arguments show that these two descriptions are completely
%equivalent, so we shall not pay further attention to the distinctions between them.

We shall say that \(\Ldthetasc_\kappa\) holds if there is a \(\theta\)-supercompactness
Laver function for \(\kappa\); in view of the definition just stated \(H_{\theta^+}\) is the default set of targets for these Laver functions, and so \(\Ldthetasc_\kappa\) is merely 
a synonym for \(\Ldthetasc_\kappa(H_{\theta^+})\). Similarly, \(\Ldmeas_\kappa\) will denote the existence of a measurable Laver function for \(\kappa\) and is a synonym for \(\Ldmeas_\kappa(H_{\kappa^+})\).
For fully supercompact cardinals \(\kappa\), the existence of a supercompactness Laver function
will be denoted by \(\Ldsc_\kappa\), which should be read more precisely as \(\Ldsc_\kappa(V)\).

While the definition of a \(\theta\)-supercompactness Laver function
refers to arbitrary \(\theta\)-supercompactness embeddings, one can in fact work solely with
embeddings arising from normal measures on \(\power_\kappa(\theta)\). This is because
any \(\theta\)-supercompactness embedding \(j\) can be factored as \(j=k\circ i\), where \(i\)
is the ultrapower by the induced normal measure on \(\power_\kappa(\theta)\) and
\(k\) is an elementary embedding with critical point strictly above \(\theta\).
Since we are only interested in guessing targets from \(H_{\theta^+}\) by \(\ell\), 
we get the same value whether we compute \(j(\ell)(\kappa)\) or \(i(\ell)(\kappa)\).
In brief, if \(\ell\) guesses a target in \(H_{\theta^+}\) via any \(\theta\)-supercompactness 
embedding, then it does so via a \(\theta\)-supercompactness ultrapower embedding. Moreover,
in the joint setting, if several \(\ell_\alpha\) guess their targets via a single 
\(\theta\)-supercompactness embedding, then they all guess their targets via a single 
\(\theta\)-supercompactness embedding as well.

Observe that there are at most \(2^\kappa\) many \(\theta\)-supercompactness
Laver functions for \(\kappa\), since there are only \(2^\kappa\) many functions
\(\kappa\to V_\kappa\). Since a joint Laver sequence cannot have the same
function appear on two different coordinates (as they could never guess two different
targets), this implies
that \(\lambda=2^\kappa\) is the largest cardinal for which 
there could possibly be a joint Laver sequence of length \(\lambda\). 
Bounding from the other side,
a single \(\theta\)-supercompactness Laver function, for some \(\theta\leq 2^\kappa\), 
already yields a joint Laver sequence of length \(\theta\).

\begin{proposition}
\label{prop:SCHasSomeJLDiamond}
If\/ \(\Ldthetasc_\kappa\) holds, then there is a \(\theta\)-supercompactness joint Laver sequence for 
\(\kappa\) of length \(\lambda=\min\{\theta,2^\kappa\}\).
\end{proposition}

\begin{proof}
Let \(\ell\) be a Laver function for \(\kappa\), and fix a subset \(I\) of 
\(\mathcal{P}(\kappa)\)
of size \(\lambda\) and a bijection \(f\colon \lambda\to I\). 
For \(\alpha<\lambda\) define \(\ell_\alpha\colon \kappa\to V_\kappa\) by
\(\ell_\alpha(\xi)=\ell(\xi)(f(\alpha)\cap\xi)\) if this makes sense and
\(\ell_\alpha(\xi)=\emptyset\) otherwise. We claim that
\(\langle \ell_\alpha;\alpha<\lambda\rangle\) is a joint Laver sequence for \(\kappa\).

To verify this, let \(\vec{a}=\langle a_\alpha;\alpha<\lambda\rangle\) be a sequence of 
elements of \(H_{\theta^+}\). Then \(\vec{a}\circ f^{-1}\in H_{\theta^+}\), so by assumption 
there is a \(\theta\)-supercompactness embedding \(j\colon V\to M\) such that 
\(j(\ell)(\kappa)=\vec{a}\circ f^{-1}\).
But now observe that, for any \(\alpha<\lambda\),
\[
j(\ell_\alpha)(\kappa)=j(\ell)(\kappa)(j(f(\alpha))\cap\kappa)=
j(\ell)(\kappa)(f(\alpha))=a_\alpha
\]
holds by elementarity.
\end{proof}

Of course, if a given Laver function works for many degrees of supercompactness
then the joint Laver sequence derived above will work for those same degrees.
In particular, if \(\kappa\) is fully supercompact then this observation, combined with
Laver's original construction, gives us a supercompactness joint Laver sequence
of length \(2^\kappa\).

\begin{corollary}
\label{cor:SCHasLongJLDiamond}
If \(\kappa\) is supercompact then \(\jLdsc_{\kappa,2^\kappa}\) holds.
\end{corollary}

It should be pointed out that Laver's original
argument only shows that a \(\theta\)-supercompactness Laver function exists for a cardinal
\(\kappa\) provided \(\kappa\) is somewhere in the range of \(2^{\theta^{<\kappa}}\)-supercompact.
Plain \(\theta\)-supercompactness does not suffice in the case \(\theta=\kappa\) (that is, in the
measurable case), since there are no Laver functions in Kunen's model \(L[U]\) (as will follow
from Proposition~\ref{prop:JLDiamondHasManyMeasures}). It is currently unknown whether
the hypothesis from Laver's proof can be reduced to just \(\theta\)-supercompactness in
the case that \(\kappa<\theta\).\footnote{We might hope that recent work in providing a canonical inner model for finite levels of supercompactness will go some way towards answering this question.}
We will not say much about this question in the present paper, and, since we are mainly interested in
jointness phenomena, we will liberally assume that the large cardinal in question carries
at least one Laver function.
 
Proposition~\ref{prop:SCHasSomeJLDiamond} essentially shows that joint Laver sequences
of maximal length
exist automatically for cardinals with a high degree of supercompactness that carry at least one Laver function. Since we will
be interested in comparing the strengths of the principles \(\jLd_{\kappa,\lambda}\)
for various \(\lambda\), we will in the remainder of this section be mostly concerned
with cardinals \(\kappa\) which are not \(2^\kappa\)-supercompact (but are at least
measurable), so as to avoid situations where a single Laver function gives rise
to the longest possible joint Laver sequence.

\subsection{Creating long joint Laver diamonds}

We now show that the existence of \(\theta\)-supercompactness joint Laver sequences
of maximal length does not require strength beyond \(\theta\)-supercompactness
itself.

The following notion is due to Hamkins~\cite{Hamkins2000:LotteryPreparation}, although its
original form, relating to strong compactness, dates back to Menas~\cite{Menas1974:OnStrongCompactnessAndSupercompactness}.

\begin{definition}
	A \(\theta\)-supercompactness \emph{Menas function} for a cardinal \(\kappa\) is a function
	\(f\colon\kappa\to\kappa\) such that there is a \(\theta\)-supercompactness embedding
	\(j\colon V\to M\) with \(\cp(j)=\kappa\) and \(j(f)(\kappa)>\theta\).
\end{definition}

A Menas function is a particularly weak form of a Laver function. If \(\ell\)
is a \(\theta\)-supercompactness Laver function for \(\kappa\), then it is also a
\(\theta\)-supercompactness Menas function, since we can pick the embedding \(j\) to have \(\ell\)
guess \(\theta+1\), for example. However, the advantage of Menas functions is that we can prove
their existence from the optimal large cardinal hypothesis on \(\kappa\), something which is
unknown for Laver functions, as we mentioned in the preceding subsection.

\begin{lemma}
\label{lemma:ThetaSCMenasFunction}
If \(\kappa\) is \(\theta\)-supercompact, then \(\kappa\) carries a \(\theta\)-supercompactness Menas function.
\end{lemma}

It is unclear who this lemma should be attributed to; we heard the following proof from Hamkins.

\begin{proof}
	We consider the nontrivial case when \(\kappa\leq\theta\).
	Define a function \(f\colon \kappa\to\kappa\) by letting \(f(\alpha)=0\) if \(\alpha\) is
	\(\kappa\)-supercompact, and \(f(\alpha)=2^{\lambda^{<\alpha}}\)
	where \(\lambda<\kappa\) is least such that \(\alpha\) is not \(\lambda\)-supercompact.
	It is simple to check that \(f\) really maps into \(\kappa\), since if \(\alpha<\kappa\)
	is \(\lambda\)-supercompact for every \(\lambda<\kappa\), then it is \(\theta\)-supercompact.
	We claim that this \(f\) is a Menas function.
	
	Let \(j\colon V\to M\) be a \(\theta\)-supercompactness embedding with critical point \(\kappa\)
	such that \(\kappa\) is not \(\theta\)-supercompact in \(M\). One can find such an embedding
	by either considering the Mitchell order on normal measures on \(\power_\kappa(\theta)\) and
	taking the embedding corresponding to a minimal such measure, or by simply taking a
	\(\theta\)-supercompactness embedding \(j\) for which \(j(\kappa)\) is least among all
	such embeddings.\footnote{In either case, the key observation is that, if \(\mu\) and \(\nu\)
	are normal measures on \(\power_\kappa(\theta)\) and \(\mu\) appears in the ultrapower by
	\(\nu\), then \(j_\mu(\kappa)<j_\nu(\kappa)\), where \(j_\mu\) and \(j_\nu\) are the 
	corresponding embeddings.}
	Let us see that \(j(f)(\kappa)>\theta\).
	
	Since \(\kappa\) is not \(\theta\)-supercompact in \(M\) and \(j(\kappa)>\theta\), we see that
	\(j(f)(\kappa)=(2^{\lambda^{<\kappa}})^M\), where \(\lambda\) is the least such that
	\(\kappa\) is not \(\lambda\)-supercompact in \(M\). 
	So suppose that \((2^{\lambda^{<\kappa}})^M\leq\theta\). Since \(M\) is closed under \(\theta\)-sequences,
	this means that \(M\) computes \(\power_\kappa(\lambda)\) and \(\power(\power_\kappa(\lambda))\) 
	correctly and, in fact, contains every subset of \(\power(\power_\kappa(\lambda))\) from \(V\).
	But \(V\) has a normal measure on \(\power_\kappa(\lambda)\), which is a subset just like that,
	and so \(M\) must also have this normal measure. This means that \(\kappa\) is in fact
	\(\lambda\)-supercompact in \(M\), contradicting our earlier assumption.
\end{proof}

\begin{theorem}
\label{thm:ForceLongJLDiamondSC}
If \(\kappa\) is \(\theta\)-supercompact, then there is a forcing extension in which
\(\jLdthetasc_{\kappa,2^\kappa}\) holds.
\end{theorem}

It should be mentioned that the forcing we do in the course of the proof may collapse \(2^\kappa\),
and so the \(\jLdthetasc_{\kappa,2^\kappa}\) in the conclusion of the theorem should be read
with the extension's version of \(2^\kappa\).

\begin{proof}
Since the \(\theta\)-supercompactness of \(\kappa\) implies its \(\theta^{<\kappa}\)-supercompactness (see, for example,~\cite[Proposition 22.11(b)]{Kanamori2005:HigherInfinite}),
we may assume that \(\theta^{<\kappa}=\theta\). Furthermore, we assume that
\(2^\theta=\theta^+\), since this may be forced without adding subsets to 
\(\power_\kappa(\theta)\) (which means that any measure on \(\power_\kappa(\theta)\) remains
a measure) or functions \(\power_\kappa(\theta)\to\theta\) (which means that any normal measure
remains normal), and so \(\kappa\) will remain \(\theta\)-supercompact after this forcing. 
Fix a Menas function \(f\) for \(\kappa\) as in Lemma~\ref{lemma:ThetaSCMenasFunction}. 
Let \(\P_\kappa\) be the length \(\kappa\) Easton support iteration which forces with 
\(\Q_\gamma=\Add(\gamma,2^\gamma)\) at inaccessible closure points of \(f\), meaning those
inaccessible \(\gamma\) for which \(f[\gamma]\subseteq\gamma\).
Finally, let \(\P=\P_\kappa*\Q_\kappa\). 
%It is useful to note that forcing with \(\P\) does not change the value of \(2^\kappa\). 
Let \(G*g\subseteq\P\) be generic; we will extract a joint Laver sequence from \(g\).

If \(g(\alpha)\) is the \(\alpha\)-th subset added by \(g\), we view it as a sequence
of bits. Using some coding scheme which admits end-of-code markers we can,
given any \(\xi<\kappa\), view the segment of \(g(\alpha)\)
between the \(\xi\)th bit and the next marker as the Mostowski code of an element of
\(V_\kappa\). We then define \(\ell_\alpha\colon\kappa\to V_\kappa\) as
follows:
given an inaccessible \(\xi\), let \(\ell_\alpha(\xi)\) be the set coded by
\(g(\alpha)\) at \(\xi\); otherwise let \(\ell_\alpha(\xi)=\emptyset\).
%given an inaccessible \(\xi\), if \(g(\alpha)\) codes a \(\P\rest (\xi+1)\)-name \(\tau\) at \(\xi\)
%let \(\ell_\alpha(\xi)=\tau^{(G*g)\rest(\xi+1)}\); otherwise let \(\ell_\alpha(\xi)=
%\emptyset\).
We claim that \(\langle \ell_\alpha;\alpha<2^\kappa\rangle\) is a joint Laver sequence.

Let \(\vec{a}=\langle a_\alpha;\alpha<2^\kappa\rangle\) be a sequence of targets in
\(H_{\theta^+}^{V[G][g]}\). 
%We can find names for these sets in \(H_{\theta^+}^V\). 
Let \(j\colon V\to M\) be the ultrapower embedding by a normal measure on 
\(\power_\kappa(\theta)\)
which corresponds to \(f\), meaning the one for which \(j(f)(\kappa)>\theta\).
We will lift this embedding through the forcing \(\P\) in \(V[G][g]\).

The argument splits into two cases, depending on the size of \(\theta\). We deal first
with the easier case when \(\theta\geq 2^\kappa\). In this case the poset 
\(j(\P_\kappa)\) factors as \(j(\P_\kappa)=\P_\kappa*\Q_\kappa*\Ptail\).
Since \(j(f)(\kappa)>\theta\), the next stage of forcing in \(j(\P_\kappa)\) above \(\kappa\)
occurs after \(\theta\), so \(\Ptail\) is \(\leq\theta\)-closed in \(M[G][g]\) and 
has size  \(j(\kappa)\) there.
It follows that \(\Ptail\) has \(j(2^\kappa)\) many subsets in \(M[G][g]\).
Since \(M\) was an ultrapower by a normal measure on \(\power_\kappa(\theta)\), 
the ordinal \(j(2^\kappa)\) has size at most \((2^\kappa)^\theta=\theta^+\) in \(V[G][g]\) (as
every smaller ordinal is represented by a function \(\power_\kappa(\theta)\to 2^\kappa\)).
Therefore, \(V[G][g]\) sees that there are only \(\theta^+\) many dense subsets of \(\Ptail\) in 
\(M[G][g]\). These can be lined up and met one at a time, using both that \(M[G][g]\) is closed
under \(\theta\)-sequences in \(V[G][g]\) and that \(\Ptail\) is a \(\leq\theta\)-closed poset in
\(M[G][g]\). This process produces in \(V[G][g]\) an 
\(M[G][g]\)-generic \(\Gtail\subseteq \Ptail\) and allows us to lift \(j\) to 
\(j\colon V[G]\to M[j(G)]\), where \(j(G)=G*g*\Gtail\).

Since \(M[j(G)]\) is still an ultrapower and thus closed under \(\theta\)-sequences in 
\(V[G][g]\),
we get \(j[g]\in M[j(G)]\). Since \(j(\Q_\kappa)\) is \(\leq\theta\)-directed closed in
\(M[j(G)]\) it has \(q=\bigcup j[g]\) as a condition. 
This \(q\) is a partial function on the domain \(j(2^\kappa)\times\kappa\).
Since \(M[j(G)]\) has both 
the sequence of targets \(\vec{a}\) and \(j\rest 2^\kappa\), we can further extend 
\(q\) to \(q^*\in M[j(G)]\) by coding \(a_\alpha\) into the bit-sequence of 
\(q(j(\alpha))\) at \(\kappa\) 
for each \(\alpha<2^\kappa\). We again diagonalize against the
\(\bigl|{2^{2^{j(\kappa)}}}\bigr|\leq \theta^+\) many dense subsets of \(j(\Q_\kappa)\) 
in \(M[j(G)]\) below the master condition \(q^*\) to get a \(M[j(G)]\)-generic \(g^*\subseteq
j(\Q_\kappa)\) and lift \(j\) to \(j\colon V[G][g]\to M[j(G)][g^*]\). Finally, observe
that we have arranged the construction of \(g^*=j(g)\) in such a way that 
\(g^*(j(\alpha))\) codes \(a_\alpha\) at \(\kappa\) 
for all \(\alpha<2^\kappa\) and, by definition,
this implies that \(j(\ell_\alpha)(\kappa)=a_\alpha\) for all \(\alpha<2^\kappa\). Thus we 
indeed have a joint Laver sequence for \(\kappa\) of length \(2^\kappa\) in \(V[G][g]\).

It remains for us to consider the second case, when \(\kappa\leq\theta<2^\kappa\). In this
situation our assumptions on \(\theta\) imply that \(2^\kappa=\theta^+\).
The poset \(j(\P_\kappa)\) factors as 
\(j(\P_\kappa)=\P_\kappa*\widetilde{\Q}_\kappa*\P_{\text{tail}}\), where 
\(\widetilde{\Q}_\kappa=\Add(\kappa,(2^\kappa)^{M[G]})\).
Since \(V[G]\) and \(M[G]\) agree on \(\power(\kappa)\), the ordinal \((2^\kappa)^{M[G]}\) has size
\(2^\kappa\) in \(V[G]\), so \(\widetilde{\Q}_\kappa\) is isomorphic, but not necessarily equal, to
\(\Q_\kappa\). Nevertheless, the same argument as before allows us to lift \(j\)
to \(j\colon V[G]\to M[j(G)]\) where \(j(G)=G*\widetilde{g}*G_{\text{tail}}\) and 
\(\widetilde{g}\) is the isomorphic image of \(g\).

We seem to hit a snag with the final lift through the forcing \(\Q_\kappa\), 
which has size \(2^\kappa\) and thus resists the usual approach of lifting via a master condition,
since this condition would simply be too big for the amount of closure we have. We salvage the
argument by using a technique, originally due to 
Magidor~\cite{Magidor1979:NonregularUltrafiltersCardinalityOfUltrapowers}, sometimes known as 
the ``master filter argument''.

The forcing \(j(\Q_\kappa)=\Add(j(\kappa),2^{j(\kappa)})^{M[j(G)]}\) has size \(2^{j(\kappa)}\) and
is \(\leq\theta\)-directed closed and \(j(\kappa)^+\)-cc in \(M[j(G)]\). Since \(M[j(G)]\)
is still an ultrapower, \(|2^{j(\kappa)}|\leq \theta^+=2^\kappa\) and so \(M[j(G)]\) has at most
\(2^\kappa\) many maximal antichains of \(j(\Q_\kappa)\), counted in \(V[G][g]\). Let these be given in the
sequence \(\langle Z_\alpha;\alpha<2^\kappa\rangle\). Since each \(Z_\alpha\) has size at most
\(j(\kappa)\), it is in fact contained in some bounded part of the poset \(j(\Q_\kappa)\). 
Furthermore (and crucially), since \(j\) is an ultrapower by a measure on \(\power_\kappa(\theta)\), it is continuous at \(2^\kappa=\theta^+\) and so there is for each \(\alpha\)
a \(\beta_\alpha<2^\kappa\) such that \(Z_\alpha\subseteq \Add(j(\kappa),j(\beta_\alpha))\).
%; for definiteness, let \(\beta_\alpha\) be the least with this property.
In particular, each \(Z_\alpha\) is a maximal antichain in \(\Add(j(\kappa),j(\beta_\alpha))\). 
We will now construct in \(V[G][g]\) a descending sequence of conditions,
deciding more and more of the antichains \(Z_\alpha\), which will generate a filter,
the ``master filter'', that will allow us to lift \(j\) to \(V[G][g]\) and also (lest we forget)
witness the joint guessing property. 
%
%\begin{claim}
%Given any dense open \(D\subseteq j(\Q_\kappa)\) and any condition \(q\) compatible with
%every condition in \(j[g]\), there is an extension \(q'\leq q\) in \(D\) which is still 
%compatible with every condition in \(j[g]\).
%\end{claim}
%
%\begin{proof}
%Let \(Z\subseteq D\) be a maximal antichain. 
%\end{proof}
%
We begin by defining the first condition \(q_0\). Consider the generic \(g\) up to
\(\beta_0\). This piece has size \(\theta\) and so \(\bigcup j[g\rest\beta_0]\) is a
condition in \(j(\Q_\kappa)\rest j(\beta_0)\). Let \(q_0'\) be the extension of
\(\bigcup j[g\rest\beta_0]\) which codes the target \(a_\alpha\) at \(\kappa\)
in \(q_0'(j(\alpha))\) for each \(\alpha<\beta_0\). This is still a condition
in \(j(\Q_\kappa)\rest j(\beta_0)\) and we can finally let \(q_0\) be any extension
of \(q_0'\) in this poset which decides the maximal antichain \(Z_0\).
Note that \(q_0\) is compatible with every condition in \(j[g]\), since we extended the
partial master condition \(\bigcup j[g\rest\beta_0]\) and made no commitments outside
\(j(\Q_\kappa)\rest j(\beta_0)\). We continue in this way recursively, constructing a descending
sequence of conditions \(q_\alpha\) for \(\alpha<\theta^+\), using the closure of
\(j(\Q_\kappa)\) and \(M[j(G)]\) to pass through limit stages. Now consider the filter
\(g^*\) generated by the conditions \(q_\alpha\). It is \(M[j(G)]\)-generic by construction
and also extends (or can easily be made to extend) \(j[g]\). We can thus lift \(j\) to
\(j\colon V[G][g]\to M[j(G)][g^*]\) and, since \(\P\) is \(\theta^+\)-cc and both 
\(G_{\text{tail}}\)
and \(g^*\) were constructed in \(V[G][g]\), the model \(M[j(G)][g^*]\) is closed under \(\theta\)-sequences
which shows that \(\kappa\) remains \(\theta\)-supercompact in \(V[G][g]\). Finally, as in
the previous case, \(g^*\) was constructed in such a way that \(j(\ell_\alpha)(\kappa)=a_\alpha\)
for all \(\alpha<2^\kappa\), verifying that these functions really do form a
joint Laver sequence for \(\kappa\).
\end{proof}

As a special case of Theorem~\ref{thm:ForceLongJLDiamondSC} we can deduce the corresponding
result for measurable cardinals.

\begin{corollary}
\label{cor:ForceLongJLDiamondMeas}
If \(\kappa\) is measurable, then there is a forcing extension in which there is a 
joint Laver sequence for \(\kappa\) of length \(2^\kappa\).
\end{corollary}

It follows from the results of Hamkins~\cite{Hamkins2003:ApproximationAndCoveringNoNewLargeCardinals} that the forcing
\(\P\) from Theorem~\ref{thm:ForceLongJLDiamondSC} does not create any measurable or
(partially) supercompact cardinals below \(\kappa\), since it admits a very low gap.
We could therefore have started with the least large cardinal \(\kappa\)
of interest and preserved its minimality throughout the construction.

\begin{corollary}
\label{cor:LeastSCCanHaveLongJLD}
If \(\kappa\) is the least \(\theta\)-supercompact cardinal, then there
is a forcing extension where \(\kappa\) remains the least \(\theta\)-supercompact cardinal
and \(\jLdthetasc_{\kappa,2^\kappa}\) holds.
\end{corollary}

It is perhaps interesting to observe the peculiar arrangement of cardinal
arithmetic in the model produced in the above proof. We have
\(2^\theta=\theta^+\) and, if \(\theta\leq 2^\kappa\), also 
\(2^\kappa=\theta^+\). In particular, we never produced a \(\theta\)-supercompactness
joint Laver sequence of length greater than \(\theta^+\) (assuming here, of course,
that \(\theta=\theta^{<\kappa}\) is the optimal degree of supercompactness).
One has to wonder whether this is significant. Certainly the existence of
long joint Laver sequences does not imply much about cardinal arithmetic, since,
for example, if \(\kappa\) is indestructibly supercompact, we can manipulate the
value of \(2^\kappa\) freely, while maintaining the existence of a supercompactness
joint Laver sequence of length \(2^\kappa\).
On the other hand, even in the case of measurable \(\kappa\), the consistency 
strength of \(2^\kappa>\kappa^+\) is
known to exceed that of \(\kappa\) being measurable. The following question is therefore
natural:
\begin{question}
If \(\kappa\) is \(\theta\)-supercompact and \(2^\kappa>\theta^+\), is there a forcing
extension preserving these facts in which there is a 
joint Laver sequence for \(\kappa\) of length \(2^\kappa\)?
\end{question}

We next show that the existence of joint Laver sequences is preserved under mild 
forcing. This will be
useful later when we separate the existence of these sequences based on their lengths.

\begin{lemma}
\label{lemma:JLDiamondSCPreservation}
Let \(\kappa\) be \(\theta\)-supercompact (with \(\kappa\leq\theta\)), \(\lambda\) a cardinal, and assume \(\jLdthetasc_{\kappa,\lambda}\) holds. Suppose \(\P\) is a poset such that either
\begin{enumerate}[label=\textup{(\arabic*)}]
\item\label{lab:distributive} \(\lambda\geq \theta^{<\kappa}\) and \(\P\) is \(\leq\lambda\)-distributive, or 
\item\label{lab:lifting} \(|\P|\leq \kappa\) and, for any \(\lambda\)-sequence of targets, some embedding associated
to the joint Laver sequence and these targets lifts through \(\P\).
\end{enumerate}
Then forcing with \(\P\) preserves \(\jLdthetasc_{\kappa,\lambda}\).
\end{lemma}

If \(|\P|\leq\kappa\), it is very often the case that every \(\theta\)-supercompactness embedding 
with critical point \(\kappa\) lifts through \(\P\), so the hypothesis in item~\ref{lab:lifting} is easily
satisfied. Furthermore, while the restriction \(|\P|\leq\kappa\) in~\ref{lab:lifting} is
necessary for full generality, it can in fact be relaxed to \(|\P|\leq\theta\) for
a large class of forcings.

\begin{proof}
Under the hypotheses of~\ref{lab:distributive} every ultrapower embedding by a measure on 
\(\power_\kappa(\theta)\) lifts to the extension by \(\P\) (see, for example, \cite[Proposition 15.1]{Cummings:Handbook}) and no elements
of \(H_{\theta^+}\) or \(\lambda\)-sequences of these are added, so any ground model
joint Laver sequence of length \(\lambda\) remains such in the extension.

Now suppose that the hypotheses of~\ref{lab:lifting} hold, let \(\langle \ell_\alpha;\alpha<\lambda\rangle\)
be a joint Laver sequence for \(\kappa\) and let \(G\subseteq \P\) be generic. We may
also assume that the underlying set of \(\P\) is a subset of \(\kappa\).
Define functions \(\ell'_\alpha\colon \kappa\to V_\kappa[G]\) by
\(\ell'_\alpha(\xi)=\ell_\alpha(\xi)^{G\cap\xi}\) if this makes sense and \(\ell'_\alpha(\xi)=
\emptyset\) otherwise. We claim that \(\langle \ell'_\alpha;\alpha<\lambda\rangle\) 
is a joint Laver sequence in \(V[G]\).

Let \(\vec{a}\) be a \(\lambda\)-sequence of targets in \(H_{\theta^+}^{V[G]}\), and let
\(\dot{\vec{a}}\) be a name for \(\vec{a}\). We can use this name and the fullness property to
derive names \(\dot{a}_\alpha\) for the targets \(a_\alpha\), uniformly in \(\alpha<\lambda\),
which means that we find the entire sequence \(\langle \dot{a}_\alpha;\alpha<\lambda\rangle\)
in \(V\).
Since each \(\dot{a}_\alpha\) names an element of \(H_{\theta^+}^{V[G]}\), and because \(\P\) is 
small (and therefore \(\theta^+\)-cc), 
we can find a nice name \(\sigma_\alpha \in H_{\theta^+}\) for each \(\dot{a}_\alpha\),
and the sequence of these nice names is in \(V\) as well.
Now let \(j\colon V\to M\) be a \(\theta\)-supercompactness embedding with critical point \(\kappa\)
which lifts through \(\P\) and which satisfies \(j(\ell_\alpha)(\kappa)=\sigma_\alpha\)
for each \(\alpha\). It then follows that \(j(\ell'_\alpha)(\kappa)=a_\alpha\) in \(V[G]\),
verifying the joint Laver diamond property there.
\end{proof}

\subsection{Separating joint Laver diamonds by length}
We next aim to show that it is consistent that there is a measurable Laver function 
for \(\kappa\) but no joint Laver sequences of length \(\kappa^+\).
The following proposition expresses the key observation
for our solution, connecting the question to the number of normal measures problem.

\begin{proposition}
\label{prop:JLDiamondHasManyMeasures}
If there is a \(\theta\)-supercompactness joint Laver sequence for \(\kappa\)
of length \(\lambda\), then there
are at least \(2^{\theta\cdot \lambda}\) many normal measures on \(\power_\kappa(\theta)\).
%In particular, there are no \(\Ldmeas\) functions in the
%canonical inner model \(L[\mu]\) for a measurable \(\kappa\).
\end{proposition}
\begin{proof}
The point is that any \(\lambda\)-sequence of targets in \(H_{\theta^+}\) can be guessed via the 
embedding arising from a normal measure on \(\power_\kappa(\theta)\) (see the discussion at the start of Section~\ref{sec:JLDSc}) and any such measure realizes a 
single \(\lambda\)-sequence via the fixed joint Laver sequence.
But there are \(2^{\theta\cdot\lambda}\) many sequences of targets, and thus there must be at least
this many measures.
\end{proof}

\begin{theorem}
\label{thm:SeparateShortLongJLDiamondMeas}
If \(\kappa\) is measurable, then there is a forcing extension in which
there is a Laver function for \(\kappa\) but no joint Laver sequence of length \(\kappa^+\).
\end{theorem}

\begin{proof}
After forcing as in the proof of Theorem~\ref{thm:ForceLongJLDiamondSC}, 
if necessary, we may assume that \(\kappa\) has a Laver
function. A result of Apter, Cummings, and Hamkins~\cite{ApterCummingsHamkins2007:LargeCardinalsFewMeasures} then shows that
\(\kappa\) still carries a Laver function in the extension by
\(\P=\Add(\omega,1)*\Coll\bigl(\kappa^+,2^{2^\kappa}\bigr)\), 
but only carries \(\kappa^+\) many normal measures there.
Proposition~\ref{prop:JLDiamondHasManyMeasures} now implies that
there cannot be a joint Laver sequence of length \(\kappa^+\) in the extension.
\end{proof}

We can push this result a bit further to get a separation between any two desired lengths
of joint Laver sequences. To state the sharpest result we need to introduce a new
notion.

\begin{definition}
Let \(\kappa\) be a large cardinal supporting a notion of Laver diamond and
\(\lambda\) a cardinal. We say that a sequence 
\(\vec{\ell}=\langle\ell_\alpha;\alpha<\lambda\rangle\) is an
\emph{almost-joint Laver sequence} for \(\kappa\) if \(\vec{\ell}\rest\gamma\) is a joint Laver sequence for \(\kappa\)
for any \(\gamma<\lambda\). We say that \(\jLd_{\kappa,<\lambda}\) holds if there is an
almost-joint Laver sequence of length \(\lambda\).
\end{definition}

\begin{theorem}
\label{thm:SeparateBetterShortLongJLDiamondMeas}
Suppose \(\kappa\) is measurable and let \(\lambda\) be a regular cardinal satisfying
\(\kappa<\lambda\leq 2^\kappa\).
If \(\jLdmeas_{\kappa,<\lambda}\) holds then there is a forcing extension
preserving this in which \(\jLdmeas_{\kappa,\lambda}\) fails.
\end{theorem}

\begin{proof}
We imitate the proof of Theorem~\ref{thm:SeparateShortLongJLDiamondMeas} but force
instead with \(\P=\Add(\omega,1)*\Coll\bigl(\lambda,2^{2^\kappa}\bigr)\).
The analysis based on \cite{ApterCummingsHamkins2007:LargeCardinalsFewMeasures} now
shows that the final extension has at most \(\lambda\) many normal measures on \(\kappa\)
and thus there can be no joint Laver sequences of length \(\lambda\) there by 
Proposition~\ref{prop:JLDiamondHasManyMeasures}. That \(\jLdmeas_{\kappa,<\lambda}\) still
holds
follows from (the proof of) Lemma~\ref{lemma:JLDiamondSCPreservation}: 
part (2) implies that, by guessing names, the \(\jLdmeas_{\kappa,<\lambda}\)-sequence
from the ground model gives rise to one in the intermediate Cohen extension.
Part (1) then shows that each of the initial segments of this sequence remains
a joint Laver sequence in the final extension.
\end{proof}

We can also extend these results to \(\theta\)-supercompact cardinals without too much
effort.

\begin{theorem}
\label{thm:SeparateShortLongJLDiamondThetaSC}
If \(\kappa\) is \(\theta\)-supercompact, \(\theta\) is regular, and 
\(\theta^{<\kappa}=\theta\), then there is a forcing extension in
which \(\Ldthetasc_{\kappa}\) holds but \(\jLdthetasc_{\kappa,\theta^+}\) fails.
\end{theorem}

Of course, the theorem is only interesting when \(\kappa\leq\theta<2^\kappa\), in which
case the given separation is best possible in view of 
Proposition~\ref{prop:SCHasSomeJLDiamond}.

\begin{proof}
We may assume by prior forcing, as in Theorem~\ref{thm:ForceLongJLDiamondSC}, that we have a 
Laver function for \(\kappa\).
We now force with
\(\P=\Add(\omega,1)*\Coll\bigl(\theta^+,2^{2^\theta}\bigr)\) to get an extension
\(V[g][G]\). By the results of \cite{ApterCummingsHamkins2007:LargeCardinalsFewMeasures},
the extension \(V[g][G]\) has at most \(\theta^+\) many normal measures on \(\power_\kappa(\theta)\)
and therefore there are no joint Laver sequences for \(\kappa\) of length \(\theta^+\)
there by Proposition~\ref{prop:JLDiamondHasManyMeasures}. It remains to see that
there is a Laver function in \(V[g][G]\). Let
\(\ell\) be a Laver function in \(V\) and define
\(\ell'\in V[g][G]\) by \(\ell'(\xi)=\ell(\xi)^g\) if \(\ell(\xi)\) is an \(\Add(\omega,1)\)-%
name and \(\ell'(\xi)=\emptyset\) otherwise. For a given \(a\in H_{\theta^+}^{V[g][G]}=
H_{\theta^+}^{V[g]}\) we can select an \(\Add(\omega,1)\)-name \(\dot{a}\in H_{\kappa^+}^V\)
and find a \(\theta\)-supercompactness embedding \(j\colon V\to M\) such that
\(j(\ell)(\kappa)=\dot{a}\). The embedding \(j\) lifts to \(j\colon V[g][G]\to M[g][j(G)]\)
since the Cohen forcing was small and the collapse forcing was \(\leq\theta\)-closed.
But then clearly \(j(\ell')(\kappa)=\dot{a}^g=a\), so \(\ell'\) is a Laver function.
\end{proof}

\begin{theorem}
Suppose \(\kappa\) is \(\theta\)-supercompact and let \(\lambda\) be a regular cardinal
satisfying \(\theta^{<\kappa}<\lambda\leq 2^\kappa\). 
If\/ \(\jLdthetasc_{\kappa,<\lambda}\) holds, then there is a forcing extension
preserving this in which \(\jLdthetasc_{\kappa,\lambda}\) fails.
\end{theorem}

\begin{proof}
The relevant forcing is \(\Add(\omega,1)*\Coll\bigl(\lambda,2^{2^{\theta^{<\kappa}}}\bigr)\).
Essentially the argument from Theorem~\ref{thm:SeparateBetterShortLongJLDiamondMeas}
then finishes the proof.
\end{proof}

Just as with joint Laver sequences, there is an upper bound on the length of an almost-joint Laver sequence.

\begin{proposition}
	The principle \(\jLdthetasc_{\kappa,<(2^\kappa)^+}\) fails for every cardinal \(\kappa\).
\end{proposition}

\begin{proof}
	Any potential \(\jLdthetasc_{\kappa,<(2^\kappa)^+}\)-sequence must necessarily have
	the same function appear on at least two coordinates. But then any initial segment of this
	sequence containing both of those coordinates cannot be joint, since it cannot guess
	distinct targets on those coordinates.
\end{proof}

A question remains about the principles \(\jLd_{\kappa,<\lambda}\), whether they are
genuinely new or whether they reduce to other principles.

\begin{question}
Let \(\kappa\) be \(\theta\)-supercompact and \(\lambda\leq 2^\kappa\). 
Is \(\jLdthetasc_{\kappa,<\lambda}\) equivalent
to \(\jLdthetasc_{\kappa,\gamma}\) holding for all \(\gamma<\lambda\)?
\end{question}

An almost-joint Laver sequence definitely gives instances of
joint Laver diamonds of each particular length \(\gamma\). The reverse implication is especially
interesting in the case when \(\lambda=\mu^+\) is a successor cardinal. This is
because simply rearranging the functions in a joint Laver sequence of length \(\mu\)
gives joint Laver sequences of any length shorter than \(\mu^+\). The question is thus
asking whether \(\jLd_{\kappa,\mu}\) suffices for \(\jLd_{\kappa,<\mu^+}\). 

An annoying feature of the models produced in the preceding theorems is that in
all of them
the least \(\lambda\) for which \(\jLd_{\kappa,\lambda}\) fails is \(\lambda=2^\kappa\). 
One has to wonder whether this is significant.

In particular, we would like an answer to the following question: is it relatively consistent
that there is a \(\theta\)-supercompact cardinal \(\kappa\), for
some \(\theta\), such that \(\Ldthetasc_\kappa\) holds and \(\jLdthetasc_{\kappa,\lambda}\)
fails for some \(\lambda<2^\kappa\)?

%\begin{question}
%\label{q:FastFailureOfjLd}
%Is it relatively consistent that there is a \(\theta\)-supercompact cardinal \(\kappa\), for
%some \(\theta\), such that \(\Ldthetasc_\kappa\) holds and \(\jLdthetasc_{\kappa,\lambda}\)
%fails for some \(\lambda\) satisfying \(\lambda<2^\kappa\)?
%\end{question}

To satisfy the listed conditions, GCH must fail at \(\kappa\) (since we must have \(\kappa<\lambda<2^\kappa\) by Proposition~\ref{prop:SCHasSomeJLDiamond}). We can therefore
expect that achieving the situation described in the question will require some
additional consistency strength.

In the case of a measurable \(\kappa\) the answer to this question is positive: we will
show in Theorem~\ref{thm:FMDoesntHaveLongjLd} that, starting from sufficient large
cardinal hypotheses, we can produce a model where \(\kappa\) is measurable and
has a measurable Laver function but no joint Laver sequences of length
\(\kappa^+<2^\kappa\). 
The proof relies on an argument
due to Friedman and Magidor~\cite{FriedmanMagidor2009:NumberNormalMeasures}
which facilitates the simultaneous control of the number of measures
at \(\kappa\) and the value of the continuum function at \(\kappa\) and \(\kappa^+\).

Let us briefly give a general setup for the argument 
of~\cite{FriedmanMagidor2009:NumberNormalMeasures} that will allow us to carry out
our intended modifications without repeating too much of the work done there.

We first recall the higher Sacks forcing, originally due to 
Kanamori~\cite{Kanamori1980:PerfectSetForcingUncountable}. Let \(\gamma\) be an inaccessible cardinal.
A condition in the poset \(\operatorname{Sacks}(\gamma)\) is a \(<\gamma\)-closed subtree \(T\) of 
\(\funcs{<\gamma}{2}\) of height \(\gamma\), such that there is a club \(C_T\subseteq\gamma\)
such that a node \(t\in T\) is a splitting node in \(T\) if and only if \(|t|\in C_T\). Stronger
conditions are given by subtrees. A generic filter for \(\operatorname{Sacks}(\gamma)\) determines
(and is, in turn, determined by) a single branch through \(\funcs{<\gamma}{2}\), a new subset
of \(\gamma\) in the extension. The forcing is \(<\gamma\)-closed and satisfies a form of fusion,
which means that it preserves \(\gamma^+\) as well.

As a slight generalization, one might consider bushier trees instead of just binary trees
as conditions. Of particular interest will be the version in which conditions
are \(<\gamma\)-closed subtrees of \(\funcs{<\gamma}{\gamma}\) of height \(\gamma\),
where each splitting node of height \(\delta\) (and there is again a club of these heights of
splitting nodes) splits into \(\delta^{++}\) many immediate successors. Call this version
of the forcing \(\operatorname{Sacks}_{\mathrm{id}^{++}}(\gamma)\).

Fix a cardinal \(\kappa\) and suppose GCH holds up to and including \(\kappa\).
Furthermore suppose that \(\kappa\) is the critical point of an
elementary embedding \(j\colon V\to M\) satisfying the following properties:
\begin{itemize}
\item \(j\) is an extender embedding, meaning that every element of \(M\) has the form
\(j(f)(\alpha)\) for some function \(f\) defined on \(\kappa\) and some \(\alpha<j(\kappa)\).
\item \((\kappa^{++})^M=\kappa^{++}\).
\item There is a function \(f\colon\kappa\to V\), such that \(j(f)(\kappa)\) is, in
\(V\) (and therefore also in \(M\)), a sequence of \(\kappa^{++}\) many disjoint
stationary subsets of \(\kappa^{++}\cap\operatorname{Cof}_{\kappa^+}\).
\end{itemize}
Given this arrangement, Friedman and Magidor define a forcing iteration \(\P\) of length
\(\kappa+1\) (with nonstationary support) 
which forces at each inaccessible stage \(\gamma\leq\kappa\) with
\(\operatorname{Sacks}^*(\gamma,\gamma^{++})*\operatorname{Sacks}_{\mathrm{id}^{++}}(\gamma)
*\operatorname{Code}(\gamma)\). Here 
%the conditions in 
%\(\operatorname{Sacks}_{\mathrm{id}^{++}}\) are perfect trees of height \(\gamma\), splitting
%on a club, and in which every splitting node of height \(\delta\) has \(\delta^{++}\)
%many successors; furthermore 
\(\operatorname{Sacks}^*(\gamma,\gamma^{++})\) is a
large product of versions of \(\operatorname{Sacks}(\gamma)\) where the splitting levels are 
restricted to the singular elements of a club and \(\operatorname{Code}(\gamma)\) is a certain
\(\leq\gamma\)-distributive notion of forcing coding information about the stage
\(\gamma\) generics into the stationary sets given by \(f(\gamma)\).

Let \(G\subseteq\P\) be generic. 
We shall use \(\P_{<\kappa}\) to denote the initial segment of the iteration \(\P\) up to stage \(\kappa\), and \(G_{<\kappa}\) will be the corresponding restriction of the generic \(G\).
In the interest of avoiding repeating the analysis of the
forcing notion given in~\cite{FriedmanMagidor2009:NumberNormalMeasures}, we list some of the properties of the extension \(V[G]\)
that we will use (but see~\cite{FriedmanMagidor2009:NumberNormalMeasures} for proofs):
\begin{enumerate}[label=\textup{(\arabic*)}]
\item\label{lab:sacksInaccessible} \(\P\) preserves cardinals and cofinalities, and increases the values of
the continuum function by at most two cardinal steps. 
In particular, any inaccessible cardinals
of \(V\) remain such in \(V[G]\).
\item\label{lab:sacksGCH} We have \(2^\kappa=\kappa^{++}\) in \(V[G]\).
\item\label{lab:sacksProperty} \(\P\) has the \(\kappa\)-Sacks property: for any function 
\(f\colon \kappa\to\mathrm{Ord}\) in \(V[G]\) there is a function \(h\in V\) such that
\(f(\alpha)\in h(\alpha)\) for all \(\alpha\) and \(|h(\alpha)|\leq\alpha^{++}\).
\item The generic \(G\) is self-encoding in a strong way: in \(V[G]\) there is a unique 
\(M\)-generic for \(j(\P)_{<j(\kappa)}\) extending \(j[G_{<\kappa}]\).
\item If \(S_\kappa\) is the generic added by 
\(\operatorname{Sacks}_{\mathrm{id}^{++}}(\kappa)\) within \(\P\), 
then \(\bigcap j[S_\kappa]\) is a \emph{tuning fork}: the union of \(\kappa^{++}\) many 
branches, all of which split off exactly at level \(\kappa\).
\item\label{lab:sacksLift} In \(V[G]\) there are exactly \(\kappa^{++}\) many \(M\)-generics for
\(j(\P)\) extending \(j[G]\), corresponding to the \(\kappa^{++}\) many branches in
\(\bigcap j[S_\kappa]\). In particular, there are exactly \(\kappa^{++}\) many lifts
\(j_\alpha\) of \(j\) to \(V[G]\), distinguished by \(j_\alpha(S_\kappa)(\kappa)=\alpha\)
for \(\alpha<\kappa^{++}\).
\end{enumerate}

\begin{proposition}
\label{prop:FMHasLd}
In the above setup, the iteration \(\P\) adds a measurable Laver function for \(\kappa\).
\end{proposition}

\begin{proof}
Let \(G\subseteq\P\) be generic. As we stated in item~\ref{lab:sacksLift}, for any \(\alpha<\kappa^{++}\) there is a
lift \(j_\alpha\colon V[G]\to M[j_\alpha(G)]\) of \(j\) such that 
\(j_\alpha(S_\kappa)(\kappa)=\alpha\), where \(S_\kappa\) is the Sacks subset of \(\kappa\)
added by the \(\kappa\)th stage of \(G\). This shows that 
\(\bar{\ell}(\gamma)=S_\kappa(\gamma)\) is a \(\kappa^{++}\)-guessing measurable Laver
function for \(\kappa\).

Note that all of the subsets of \(\kappa\) in \(M[j_\alpha(G)]\) (and \(V[G]\)) 
appear already in \(M[G]\); this is because the part of \(j(\P)\) above stage \(\kappa\) is
forced to be \(\leq\kappa\)-distributive. 
Let \(\vec{e}=\langle e_\alpha;\alpha<\kappa^{++}\rangle\) be
an enumeration of \(H_{\kappa^+}^{M[G]}\) in \(M[G]\) and let \(\dot{e}\in M\) be
a name for \(\vec{e}\). We can write \(\dot{e}=j(F)(\kappa)\) for some function \(F\),
defined on \(\kappa\). Now define a function \(\ell\colon\kappa\to V_\kappa\) in \(V[G]\) by
\(\ell(\gamma)=(F(\gamma)^G)(\bar{\ell}(\gamma))\). This is, in fact, our desired Laver
function; given an arbitrary element of \(H_{\kappa^+}^{V[G]}\), we can find it in the 
enumeration \(\vec{e}\). If \(\alpha\) is its index, then
\[
j_\alpha(\ell)(\kappa) =
\bigl(j_\alpha(F)(\kappa)^{j_\alpha(G)}\bigr) (j_\alpha(\bar{\ell})(\kappa)) =
\bigl(j(F)(\kappa)^{j_\alpha(G)}\bigr) (\alpha) =
\vec{e}(\alpha)= e_\alpha\,.
\qedhere
\]
\end{proof}

\begin{theorem}
\label{thm:FMDoesntHaveLongjLd}
Suppose \(\kappa\) is \((\kappa+2)\)-strong and assume that \(V=L[\vec{E}]\) is the
minimal extender model witnessing this. Then there is a forcing extension in which
\(2^\kappa=\kappa^{++}\), the cardinal \(\kappa\) remains measurable, 
\(\kappa\) carries a measurable Laver function, and there are no measurable joint Laver
sequences for \(\kappa\) of length \(\kappa^+\).
\end{theorem}

%This finally answers Question~\ref{q:FastFailureOfjLd} in the positive.

\begin{proof}
Let \(j\colon V\to M\) be the ultrapower embedding by the top extender of \(\vec{E}\),
the unique extender witnessing the \((\kappa+2)\)-strongness of \(\kappa\). In particular,
every element of \(M\) has the form \(j(f)(\alpha)\) for some \(\alpha<j(\kappa)\), and
\(M\) computes \(\kappa^{++}\) correctly. Furthermore, \(V\) has a canonical
\(\diamond_{\kappa^{++}}(\operatorname{Cof}_{\kappa^+})\)-sequence, which is definable without parameters over \(H_{\kappa^{++}}\) via the standard condensation argument. 
Since \(H_{\kappa^{++}}\in M\), this same sequence is also in \(M\)
and is of the form \(j(\bar{f})(\kappa)\) for some function \(\bar{f}\), since it is definable
in \(M\) just from the parameter \(\kappa\).
By having this diamond sequence guess the singletons \(\{\xi\}\) for \(\xi<\kappa^{++}\),
we obtain a sequence of \(\kappa^{++}\) many disjoint stationary subsets of
\(\kappa^{++}\cap\operatorname{Cof}_{\kappa^+}\), and this sequence itself has the form
\(j(f)(\kappa)\) for some function \(f\). We are therefore in a situation where
the definition of the Friedman--Magidor iteration we described above makes sense.
But first, we shall carry out some preliminary forcing.

Let \(g\subseteq\Add(\kappa^+,\kappa^{+3})\) be generic. Since this Cohen poset is
\(\leq\kappa\)-distributive, the embedding \(j\) lifts (uniquely) to an embedding
\(j\colon V[g]\to M[j(g)]\).\footnote{The lifted embedding will not be
a \((\kappa+2)\)-strongness embedding and, in fact, \(\kappa\) is no longer 
\((\kappa+2)\)-strong in \(V[g]\). Nevertheless, the residue of strongness will suffice
for our argument.}
Let us examine the lifted embedding \(j\). It is still an extender embedding.
Additionally, since GCH holds in \(V\), the forcing \(\Add(\kappa^+,\kappa^{+3})\)
preserves cardinals, cofinalities, and stationary subsets of \(\kappa^{++}\).
Together, this means that \(M[j(g)]\) computes \(\kappa^{++}\) correctly, and the stationary
sets given by the sequence \(j(f)(\kappa)\) above remain stationary. Therefore we may
still define the Friedman--Magidor iteration \(\P\) over \(V[g]\).

Let \(G\subseteq\P\) be generic over \(V[g]\). We claim that \(V[g][G]\) is the model
we want. We have \(2^\kappa=\kappa^{++}\) in the extension, by item~\ref{lab:sacksGCH} of our list, and
Proposition~\ref{prop:FMHasLd} implies that \(\kappa\) is
measurable in \(V[g][G]\) and \(\Ldmeas_\kappa\) holds there. So it remains for us to see
that \(\jLdmeas_{\kappa,\kappa^+}\) fails. By Proposition~\ref{prop:JLDiamondHasManyMeasures}
it suffices to show that \(\kappa\) does not carry \(2^{\kappa^+}=\kappa^{+3}\) many
normal measures in \(V[g][G]\).

Let \(U^*\in V[g][G]\) be a normal measure on \(\kappa\) and let
\(j^*\colon V[g][G]\to N[j^*(g)][j^*(G)]\) be its associated ultrapower embedding.
This embedding restricts to \(j^*\rest V\colon V\to N\). Since \(V\) is the core model
from the point of view of \(V[g][G]\), the 
embedding \(j^*\rest V\) arises as the ultrapower map associated to a normal iteration
of extenders on the sequence \(\vec{E}\) (see~\cite[Section 7.4]{Zeman2002:InnerModels} for more
details).

We first claim that the first extender applied in this iteration is the top extender of
\(\vec{E}\). Let us write \(j^*\rest V=j_1\circ j_0\), where \(j_0\colon V\to N_0\) results from the
first applied extender. Clearly \(j_0\) has critical point \(\kappa\). 
Now suppose that \(j_0(\kappa)<\kappa^{++}\). Of course,
\(j_0(\kappa)\) is inaccessible in \(N_0\) and, since \(N\) is an inner model of \(N_0\),
also in \(N\). But \(j_0(\kappa)\) is not inaccessible in \(N[j^*(g)][j^*(G)]\), since
\(2^\kappa=\kappa^{++}\) there. This is a contradiction, since passing from \(N\)
to \(N[j^*(g)][j^*(G)]\) preserves inaccessibility, by item~\ref{lab:sacksInaccessible} of our list.

It follows that we must have \(j_0(\kappa)\geq\kappa^{++}\). We will argue that
the extender \(E\) applied to get \(j_0\) witnesses the \((\kappa+2)\)-strongness of 
\(\kappa\), so it must be the top extender of \(\vec{E}\). 
Using a suitable indexing of
\(\vec{E}\), the extender \(E\) has index \((j_0(\kappa)^+)^{N_0}>\kappa^{++}\),
and the coherence of the extender sequence implies that the sequences in \(V\) and in
\(N_0\) agree up to \(\kappa^{++}\). By the acceptability of these extender models
it now follows that \(H_{\kappa^{++}}^V=H_{\kappa^{++}}^{N_0}\) or, equivalently,
\(V_{\kappa+2}\in N_0\). This means that \(j_0\) is the \((\kappa+2)\)-strongness ultrapower of \(V\),
and is equal to the embedding \(j\) we started with.

Finally, we claim that the iteration giving rise to \(j^*\rest V\) ends after one step, meaning
that \(j^*\rest V=j_0=j\). Suppose to the contrary that \(j_1\) is nontrivial. By the normality of
the iteration, the critical point of \(j_1\) must be some \(\lambda>\kappa\).
We can find a function \(h\in V[g][G]\), defined on \(\kappa\), such that
\(j^*(h)(\kappa)=\lambda\), since \(j^*\) is given by an ultrapower of
\(V[g][G]\) by a normal measure on \(\kappa\). 
By the \(\kappa\)-Sacks property of \(\P\) (see item~\ref{lab:sacksProperty}) we can cover the function
\(h\) by a function \(\bar{h}\in V[g]\); in fact, since the forcing to add \(g\)
was \(\leq\kappa\)-closed, we have \(\bar{h}\in V\).
Now
\[
\lambda=j^*(h)(\kappa)\in j^*(\bar{h})(\kappa) = j_1(j_0(\bar{h}))(j_1(\kappa))
= j_1(j_0(\bar{h})(\kappa))
\]
and \(A=j_0(\bar{h})(\kappa)\) has cardinality at most \(\kappa^{++}\) in \(M\), using the
properties of \(\bar{h}\).
In particular, since \(\kappa^{++}<\lambda\), we have \(\lambda\in j_1(A)=j_1[A]\),
which is a contradiction, since \(\lambda\) was the critical point of \(j_1\).

We can conclude that any embedding \(j^*\) arising from a normal measure on \(\kappa\)
in \(V[g][G]\) is a lift of the ground model \((\kappa+2)\)-strongness embedding \(j\).
But there are exactly \(\kappa^{++}\) many such lifts: the lift to \(V[g]\) is unique,
and there are \(\kappa^{++}\) many possibilities for the final lift to \(V[g][G]\), according
to item~\ref{lab:sacksLift}.
Therefore there are only \(\kappa^{++}\) many normal measures on \(\kappa\)
in \(V[g][G]\).
\end{proof}

Ben-Neria and Gitik~\cite{BenNeriaGitik:UniqueMeasureWithoutGCH} have announced that the consistency strength required to
achieve the failure of GCH at a measurable cardinal carrying a unique normal measure
is exactly that of a measurable cardinal \(\kappa\) with \(o(\kappa)=\kappa^{++}\). Their method is flexible enough
to allow us to incorporate it into our proof of theorem~\ref{thm:FMDoesntHaveLongjLd},
reducing the consistency strength hypothesis required there from a \((\kappa+2)\)-strong
cardinal \(\kappa\) to just \(o(\kappa)=\kappa^{++}\). We have chosen to present the proof
based on the original Friedman--Magidor argument since it avoids some complications
arising from using the optimal hypotheses

\subsection{(Joint) Laver diamonds and the number of normal measures}

The only method of controlling the existence of (joint) Laver diamonds we have seen
is by controlling the number of large cardinal measures, relying on the rough bound
given by Proposition~\ref{prop:JLDiamondHasManyMeasures}. One has to wonder whether
merely the existence of sufficiently many measures guarantees the existence of
(joint) Laver diamonds. We focus on the simplest form of the question, concerning
measurable cardinals.

\begin{question}
	\label{q:MeasuresGiveLavers}
	Suppose \(\kappa\) is measurable and there are at least \(2^\lambda\) many normal measures
	on \(\kappa\) for some \(\lambda\geq\kappa\). 
	%Does \(\jLd_{\kappa,\lambda}\) hold?
	Does there exist a measurable joint Laver sequence for \(\kappa\) of length \(\lambda\)?
\end{question}

As the special case when \(\lambda=\kappa\), the question includes the possibility
that having \(2^\kappa\) many normal measures, the minimum required, suffices to give
the existence of a measurable Laver function for \(\kappa\).
Even in this very special case
it seems implausible that simply having enough measures
would automatically yield a Laver function. Nevertheless, in all of the examples
of models obtained by forcing and in which we have control over the number of measures
that we have seen, Laver functions have existed. 
On the other hand, Laver functions and joint Laver sequences also exist in
canonical inner models that have sufficiently many measures. These models carry
long Mitchell-increasing sequences of normal measures that we can use to obtain
ordinal-guessing Laver functions. We can then turn these into actual Laver functions
by exploiting the structure of these models.

\begin{definition}
	Let \(A\) be a set (or class) of ordinals and let \(\bar{\ell}\) be an \(A\)-guessing
	Laver function for some large cardinal \(\kappa\). Let \(\triangleleft\) be some
	wellorder (one arising from an \(L\)-like inner model, for example). 
	We say that \(\triangleleft\) is \emph{suitable} for \(\bar{\ell}\) if,
	for any \(\alpha\in A\), there is an elementary embedding \(j\), witnessing the largeness of
	\(\kappa\), such that \(\bar{\ell}\) guesses \(\alpha\) via \(j\) and \(j(\triangleleft)\rest 
	(\alpha+1) = \triangleleft \rest (\alpha+1)\); that is, the wellorders \(j(\triangleleft)\)
	and \(\triangleleft\) agree on their first \(\alpha+1\) many elements.
	
	If \(\mathcal{J}\) is a class of elementary embeddings witnessing the largeness of 
	\(\kappa\),
	we say that \(\triangleleft\) is \emph{supersuitable for \(\mathcal{J}\)} if 
	\(j(\triangleleft)\rest j(\kappa)=\triangleleft\rest j(\kappa)\) for any \(j\in\mathcal{J}\).
\end{definition}

We could, for example, take the class \(\mathcal{J}\) to consist of all ultrapower embeddings
by normal measures on \(\kappa\) or, more to the point, all ultrapower embeddings arising
from a fixed family of extenders. We should also note that, for the notion to make sense,
the order type of \(\triangleleft\) must be quite high: at least \(\sup A\) in the case
of wellorders suitable for an \(A\)-guessing Laver function and at least
\(\sup_{j\in\mathcal{J}} j(\kappa)\) for supersuitable wellorders (the latter would also
make sense if the order type of \(\triangleleft\) were smaller than \(\kappa\), but that
case is not of much interest).

%Clearly any supersuitable wellorder is suitable for any ordinal-guessing Laver function
%\(\bar{\ell}\), provided that the class \(\mathcal{J}\) includes the embeddings via which
%\(\bar{\ell}\) guesses its targets. 
If \(\bar{\ell}\) is an ordinal-guessing Laver function and \(\mathcal{J}\) is a class of
elementary embeddings that includes all the embeddings that \(\bar{\ell}\) requires to guess
its targets, then any wellorder that is supersuitable for \(\mathcal{J}\) is also
suitable for \(\bar{\ell}\).
The following lemma describes the way in which 
suitable wellorders will be used to turn ordinal-guessing Laver functions into set-guessing 
ones.

\begin{lemma}
	\label{lemma:SetGuessingFromOrdinalGuessing}
	Let \(A\) be a set (or class) of ordinals and let \(\bar{\ell}\) be an
	\(A\)-guessing Laver function for some large cardinal \(\kappa\). Let \(\triangleleft\)
	be a wellorder such that \(\operatorname{otp}(\triangleleft)\subseteq A\).
	If \(\triangleleft\) is suitable for \(\bar{\ell}\), then there is a \(B\)-guessing Laver 
	function for \(\kappa\), where \(B\) is the field of \(\triangleleft\).
\end{lemma}

\begin{proof}
	We can define a \(B\)-guessing Laver function by simply letting
	\(\ell(\xi)\) be the \(\bar{\ell}(\xi)\)th element of \(\triangleleft\). Then,
	given a target \(b\in B\), we can find its index \(\alpha\) in the wellorder 
	\(\triangleleft\) and an embedding \(j\) such that \(j(\bar{\ell})(\kappa)=\alpha\).
	Since \(\triangleleft\) is suitable for \(\bar{\ell}\), the orders \(\triangleleft\) and
	\(j(\triangleleft)\) agree on their \(\alpha\)th element and so \(\ell\) guesses \(b\)
	via \(j\).
\end{proof}

It follows from the above lemma that in any model with a sufficiently supersuitable
wellorder, being able to guess ordinals suffices to be able to guess arbitrary sets.

\begin{lemma}
	Let \(X\) be a set (or class) of ordinals and let \(\mathcal{J}\) be a class of elementary
	embeddings of \(L[X]\) with critical point \(\kappa\) such that 
	\(j(X)\cap j(\kappa) = X\cap j(\kappa)\) for any
	\(j\in\mathcal{J}\).
	Then \(\leq_X\), the canonical order of \(L[X]\), is supersuitable for \(\mathcal{J}\).
\end{lemma}

\begin{proof}
	This is obvious; the order \(\leq_X\rest j(\kappa)\) is definable in \(L_{j(\kappa)}[X]\),
	but by our coherence hypothesis this structure is just the same as \(L_{j(\kappa)}[j(X)]\).
\end{proof}

We are mostly interested in this lemma in the case when \(X=\vec{E}\) is an extender sequence
and \(L[\vec{E}]\) is an extender model in the sense of~\cite{Zeman2002:InnerModels}.
In particular, we want \(\vec{E}\) to be \emph{acceptable} (a technical condition which 
implies enough condensation properties in \(L[\vec{E}]\) to conclude
\(H_\lambda^{L[\vec{E}]}=L_\lambda[\vec{E}]\)), \emph{coherent} (meaning that if
\(j\colon L[\vec{E}]\to L[\vec{F}]\) is an ultrapower by the \(\alpha\)th extender of
\(\vec{E}\) then \(\vec{F}\rest (\alpha+1)= \vec{E}\rest\alpha\)), and to use Jensen indexing
(meaning that the index of an extender \(E\) on \(\vec{E}\) with critical point \(\kappa\) 
is \(j_E(\kappa)^+\), as computed in the ultrapower).

\begin{corollary}
	\label{cor:SupersuitableOrderInExtenderModels}
	Let \(V=L[\vec{E}]\) be an extender model.
	Then the canonical wellorder is supersuitable for the class of ultrapower embeddings
	by the extenders on the sequence \(\vec{E}\).
\end{corollary}

\begin{proof}
	This is immediate from the preceding lemma and the fact that our extender sequences are
	coherent and use Jensen indexing.
\end{proof}

\begin{theorem}
	\label{thm:LaverFunctionsInExtenderModels}
	Let \(V=L[\vec{E}]\) be an extender model. Let \(\kappa\) be a
	cardinal such that every normal measure on \(\kappa\) appears on the sequence \(\vec{E}\). 
	If \(o(\kappa)\geq\kappa^+\) then \(\Ldmeas_\kappa\) holds.
	Moreover, if \(o(\kappa)=\kappa^{++}\) then \(\jLdmeas_{\kappa,\kappa^+}\),
	and even \(\Ldmeas_\kappa(H_{\kappa^{++}})\), holds.
\end{theorem}

In particular, the above theorem implies \(\Ldmeas_\kappa\) holds in the least inner model
with the required number of measures and the same holds for \(\jLdmeas_{\kappa,\kappa^+}\).
This provides further evidence that the answer to Question~\ref{q:MeasuresGiveLavers}, which
remains open, might turn out to be positive.

\begin{proof}
	We can argue for the two cases more or less uniformly: let 
	\(\lambda\in\{\kappa^+,\kappa^{++}\}\) such that \(\lambda\leq o(\kappa)\).
	The function \(\bar{\ell}(\xi)=o(\xi)\) is a \(\lambda\)-guessing measurable Laver
	function for \(\kappa\).  By the acceptability of
	\(\vec{E}\) we have that \(H_{\lambda}=L_{\lambda}[\vec{E}]\). The canonical wellorder
	\(\leq_{\vec{E}}\cap L_\lambda[\vec{E}]\) has order type \(\lambda\) and, by
	Corollary~\ref{cor:SupersuitableOrderInExtenderModels}, is supersuitable for the class
	of ultrapower embeddings by normal measures on \(\kappa\). It follows that
	\(\leq_{\vec{E}}\) is suitable for \(\bar{\ell}\), so, by 
	Lemma~\ref{lemma:SetGuessingFromOrdinalGuessing}, there is an \(H_{\lambda}\)-guessing
	measurable Laver function for \(\kappa\).
	
	To finish the proof we still need to produce a joint measurable Laver sequence
	for \(\kappa\), in the case that \(o(\kappa)=\kappa^{++}\). This is done in exactly the same
	way as in Proposition~\ref{prop:SCHasSomeJLDiamond}; one simply uses the 
	\(H_{\kappa^{++}}\)-guessing Laver function to guess the whole sequence of targets
	for a joint Laver sequence.
\end{proof}

Interestingly, if we restrict to a smaller set of targets, having
enough normal measures \emph{does} give us Laver functions.

\begin{lemma}
	\label{lemma:DiscreteMeasures}
	Let \(\kappa\) be a regular cardinal and \(\gamma\leq\kappa\)
	and suppose that \(\langle \mu_\alpha;\alpha<\gamma\rangle\) is a sequence of distinct normal
	measures on \(\kappa\). Then there is a sequence \(\langle X_\alpha;\alpha<\gamma\rangle\)
	of pairwise disjoint subsets of \(\kappa\) such that \(X_\alpha\in\mu_\beta\) if
	and only if \(\alpha=\beta\).
\end{lemma}

\begin{proof}
	We prove the lemma by induction on \(\gamma\). In the base step, \(\gamma=1\),
	we simply observe that, since \(\mu_0\neq\mu_1\), we must have a set \(X_0\subseteq\kappa\)
	such that \(X_0\in\mu_0\) and \(\kappa\setminus X_0\in \mu_1\).
	
	The successor step proceeds similarly. 
	Suppose that the lemma holds for sequences of length \(\gamma\) and fix a sequence
	of measures \(\langle \mu_\alpha;\alpha<\gamma+1\rangle\). By the induction hypothesis
	we can find pairwise disjoint sets \(\langle Y_\alpha;\alpha<\gamma\rangle\) such
	that each \(Y_\alpha\) picks out a unique measure among those with indices below
	\(\gamma\). Again, since \(\mu_\gamma\) is distinct from all of the other measures,
	we can find sets \(Z_\alpha\in \mu_\gamma\setminus \mu_\alpha\) for each \(\alpha<\gamma\).
	Then the sets \(X_\alpha=Y_\alpha\setminus Z_\alpha\) for \(\alpha<\gamma\) and
	\(X_\gamma=\bigcap_{\alpha<\gamma} Z_\alpha\) are as required.
	
	In the limit step suppose that the lemma holds for all \(\delta<\gamma\).
	We can then fix sequences \(\langle X_\alpha^\delta;\alpha<\delta\rangle\) for each 
	\(\delta<\gamma\)
	as above. The argument proceeds slightly differently depending on whether \(\gamma=\kappa\)
	or not. If \(\gamma<\kappa\) we can simply let 
	\(X_\alpha=\bigcap_{\alpha<\delta<\gamma} X_\alpha^\delta\in\mu_\alpha\). 
	If, on the other hand, we have \(\gamma=\kappa\), then first let
	\(Y_\alpha=\diag_{\alpha<\delta<\kappa}X_\alpha^\delta\in \mu_\alpha\).
%	Observe that the \(Y_\alpha\) are almost disjoint: \(Y_\alpha\cap Y_\beta\) is
%	bounded in \(\kappa\) for any \(\alpha,\beta<\kappa\).
	Given \(\alpha<\beta<\gamma\), the sets \(Y_\alpha\) and \(Y_\beta\) are almost contained
	in \(X_\alpha^{\beta+1}\) and \(X_\beta^{\beta+1}\), respectively. Since these two are, in
	turn, disjoint, \(Y_\alpha\) and \(Y_\beta\) have bounded intersection.
	Now consider 
	\[X_\alpha=Y_\alpha\setminus \bigcup_{\beta<\alpha}(Y_\alpha\cap Y_\beta)\]
	for \(\alpha<\kappa\). Since \(Y_\alpha\cap Y_\beta\) is bounded for all \(\beta<\alpha\),
	we still have \(X_\alpha\in \mu_\alpha\). Furthermore, we obviously have
	\(X_\alpha\cap X_\beta=\emptyset\) for \(\beta<\alpha\) and this implies that the
	\(X_\alpha\) are pairwise disjoint.
\end{proof}

\begin{theorem}
	\label{thm:OrdGuessingIffEnoughMeasures}
	Let \(\kappa\) be a measurable cardinal and \(\gamma<\kappa^+\) an ordinal.
	There is a \(\gamma\)-guessing measurable Laver function for \(\kappa\) if and only if
	there are at least \(|\gamma|\) many normal measures on \(\kappa\).
\end{theorem}

\begin{proof}
	First suppose that \(\Ldmeas_\kappa(\gamma)\) holds. Then, just as in
	Proposition~\ref{prop:JLDiamondHasManyMeasures}, each target \(\alpha<\gamma\)
	requires its own embedding \(j\) via which it is guessed and this gives us
	\(|\gamma|\) many distinct normal measures.
	
	Conversely, suppose that we have at least \(|\gamma|\) many normal measures on
	\(\kappa\). We can apply Lemma~\ref{lemma:DiscreteMeasures} to find a sequence
	of pairwise disjoint subsets of \(\kappa\) distinguishing these measures. By reorganizing
	the measures and the distinguishing sets we may assume that they are given in sequences
	of length \(\gamma\). We now have normal measures \(\langle \mu_\alpha;\alpha<\gamma\rangle\)
	and sets \(\langle X_\alpha;\alpha<\gamma\rangle\) such that \(\mu_\alpha\) is the unique
	measure concentrating on \(X_\alpha\); we may even assume that the \(X_\alpha\) partition
	\(\kappa\). 
	
	Let \(f_\alpha\) for
	\(\alpha<\gamma\) be the representing functions for \(\alpha\), that is,
	\(j(f_\alpha)(\kappa)=\alpha\) for any ultrapower embedding \(j\) by a normal measure on
	\(\kappa\). Constructing these functions is not difficult. If \(\alpha<\kappa\), we can simply
	take \(f_\alpha\) to be the constant function with value \(\alpha\). If 
	\(\kappa\leq\alpha<\gamma<\kappa^+\), we can fix a wellorder \(\triangleleft_\alpha\) of
	\(\kappa\) in ordertype \(\alpha\), and the function \(f_\alpha(\xi)=\operatorname{otp}(\triangleleft_\alpha\cap(\xi\times\xi))\)
	will be a representing function for \(\alpha\).\footnote{These functions \(f_\alpha\) are essentially just the first \(\kappa^+\) many canonical functions for \(\kappa\), see~\cite[Section 1.3]{Jech:Handbook}.}
	
	We can now define a \(\gamma\)-guessing Laver function \(\ell\) by letting
	\(\ell(\xi)=f_\alpha(\xi)\) where \(\alpha\) is the unique index such that 
	\(\xi\in X_\alpha\). This function indeed guesses any target \(\alpha<\gamma\):
	simply let \(j\colon V\to M\) be the ultrapower by \(\mu_\alpha\).
	Since \(\mu_\alpha\) concentrates on \(X_\alpha\) we have 
	\(j(\ell)(\kappa)=j(f_\alpha)(\kappa)=\alpha\).
\end{proof}

\begin{corollary}
	\label{cor:SmallGuessingIffEnoughMeasures}
	Let \(\kappa\) be a measurable cardinal and fix a subset \(A\subseteq H_{\kappa^+}\)
	of size at most \(\kappa\). Then there is an \(A\)-guessing measurable Laver function for
	\(\kappa\) if and only if there are at least \(|A|\) many normal measures on \(\kappa\).
\end{corollary}

\begin{proof}
	The forward direction follows just as before: each target in \(A\) gives its own
	normal measure on \(\kappa\). Conversely, if there are at least \(|A|\) many normal
	measures on \(\kappa\) then, by Theorem~\ref{thm:OrdGuessingIffEnoughMeasures},
	there is an \(|A|\)-guessing measurable Laver function \(\bar{\ell}\). Fix a bijection
	\(f\colon |A|\to A\). We may assume, moreover, that \(A\subseteq\mathcal{P}(\kappa)\). 
	Then we can define an \(A\)-guessing Laver function \(\ell\) by letting 
	\(\ell(\xi)=f(\bar{\ell}(\xi))\cap\xi\). This works: to guess \(f(\alpha)\) we let
	\(\bar{\ell}\) guess \(\alpha\) via some \(j\). Then 
	\(j(\ell)(\kappa)=j(f(\alpha))\cap\kappa=f(\alpha)\).
\end{proof}

Lemma~\ref{lemma:DiscreteMeasures} can be recast in somewhat different language,
giving it, and the subsequent results, a more topological flavour.

Given a cardinal \(\kappa\), let \(\mathcal{M}(\kappa)\) be the set of normal measures on
\(\kappa\). We can topologize \(\mathcal{M}(\kappa)\) by having, for each 
\(X\subseteq\kappa\), a basic neighbourhood \([X]=\set{\mu\in \mathcal{M}(\kappa)}{X\in\mu}\)
(this is just the topology induced on \(\mathcal{M}(\kappa)\) by the Stone topology on
the space of ultrafilters on \(\kappa\)). Lemma~\ref{lemma:DiscreteMeasures} can now
be restated to say that any subspace of \(\mathcal{M}(\kappa)\) of size at most \(\kappa\)
is discrete and, moreover, the basic open sets
witnessing this can be taken to arise from a pairwise disjoint family of subsets of \(\kappa\);
such subspaces of spaces of ultrafilters are sometimes also called \emph{strongly discrete} (see~\cite{Dow1985:GoodOKUltrafilters}, for example). 
One might thus hope to show
the existence of Laver functions by exhibiting even larger discrete subspaces of
\(\mathcal{M}(\kappa)\). In pursuit of that goal we obtain the following generalization of 
Corollary~\ref{cor:SmallGuessingIffEnoughMeasures}.

\begin{theorem}
	\label{thm:equivalenceOfLaverAndAlmostStronglyDisjointMeasures}
	Let \(\kappa\) be a measurable cardinal and \(A\subseteq\mathcal{P}(\kappa)\). 
	Then \(\Ldmeas_\kappa(A)\) holds if and only if
	there are for each \(a\in A\) a set \(S_a\subseteq\kappa\) and a normal
	measure \(\mu_a\) on \(\kappa\) such that \(\set{\mu_a}{a\in A}\) is discrete in \(\mathcal{M}(\kappa)\),
	as witnessed by \(\set{S_a}{a\in A}\), and \(S_a\cap S_b\subseteq \set{\xi}{a\cap\xi=b\cap\xi}\).
\end{theorem}

\noindent We could have relaxed our hypothesis to \(A\subseteq H_{\kappa^+}\)
by working with Mostowski codes.

\begin{proof}
	Assume first that \(\ell\) is a measurable \(A\)-guessing Laver function for \(\kappa\).
	Then we can let \(S_a=\set{\xi}{\ell(\xi)=a\cap\xi}\). Obviously we have \(j(\ell)(\kappa)=a\)
	if and only if the measure derived from \(j\) concentrates on \(S_a\). It follows
	that the measures \(\mu_a\) derived this way form a discrete subspace of
	\(\mathcal{M}(\kappa)\) and we obviously have
	\(S_a\cap S_b\subseteq \set{\xi}{a\cap \xi=b\cap\xi}\).
	
	Conversely, assume we have such a discrete family of measures \(\mu_a\) 
	and a family of sets \(S_a\) as described. We can define an \(A\)-guessing
	measurable Laver function \(\ell\) by letting \(\ell(\xi)=a\cap \xi\) where
	\(a\) is such that \(\xi\in S_a\). This is well defined by the coherence condition
	imposed upon the \(S_a\), and it is easy to see that \(\ell\) satisfies the guessing property.
\end{proof}

This topological viewpoint presents a number of questions which might suggest an
approach to Question~\ref{q:MeasuresGiveLavers}. 
%For example, 
%it is unclear whether, given a discrete family of normal measures one can find an almost
%disjoint discretizing family as in the above theorem. Even more pressingly, we do not know
%whether it is possible for \(\mathcal{M}(\kappa)\) to have no discrete subspaces of size
%\(\kappa^+\) (while itself having size at least \(\kappa^+\)).
For example, it might be the case that \emph{every} discrete subset of
\(\mathcal{M}(\kappa)\) has its discreteness witnessed by a family of sets \(S_a\) as in
Theorem~\ref{thm:equivalenceOfLaverAndAlmostStronglyDisjointMeasures}. If this were so, we could
reduce the problem of finding Laver functions to the seemingly simpler problem of finding large\footnote{Recall that Lemma~\ref{lemma:DiscreteMeasures} says that every subset of \(\mathcal{M}(\kappa)\) of size \(\leq\kappa\) is already strongly discrete.}
discrete subspaces of \(\mathcal{M}(\kappa)\). But even this simpler task is problematic, since
it might be possible that \(\mathcal{M}(\kappa)\) has size (at least) \(\kappa^+\) but has no
discrete subspaces of size \(\kappa^+\) at all.

\subsection{\(\protect\Ld_\kappa\)-trees}

Thus far we have thought of joint Laver diamonds as simply matrices or sequences of Laver
diamonds. To better facilitate the reflection properties required for forcing
iterations using prediction, we would now like a different representation.
A reasonable attempt seems to be trying to align the joint Laver sequence
with the full binary tree of height \(\kappa\).

\begin{definition}
	%A joint Laver sequence \(\langle \ell_\alpha;\alpha<2^\kappa\rangle\)
	%is \emph{fully treeable} if there is a bijection \(f\colon \funcs{\kappa}{2}\to 2^\kappa\)
	%such that \(\ell_{f(s)}(\xi)=\ell_{f(t)}(\xi)\) whenever \(s\rest(\xi+1)= t\rest(\xi+1)\).
	Let \(\kappa\) be a large cardinal supporting a notion of Laver diamond.
	A \emph{\(\Ld_\kappa\)-tree} is a labelling of the binary tree such that the labels along the
	branches form a joint Laver sequence. More precisely, a \(\Ld_\kappa\)-tree is a function 
	\(D\colon \funcs{<\kappa}{2}\to V\)
	such that for any sequence of targets \(\langle a_s;s\in\funcs{\kappa}{2}\rangle\)
	there is an elementary embedding \(j\), witnessing the largeness \(\kappa\),
	such that \(j(D)(s)=a_s\) for all \(s\in \funcs{\kappa}{2}\).
	
	Given an \(I\subseteq\funcs{\kappa}{2}\), an \emph{\ILd-tree} is
	a function \(D\) as above, satisfying the same guessing property but only for sequences of
	targets indexed by \(I\).
\end{definition}

Naturally, we will in this section mostly be interested in \(\Ldthetasc_\kappa\)-trees, that is,
\(\Ld_\kappa\)-trees whose branches form a \(\jLdthetasc_{\kappa,2^\kappa}\)-sequence.
If the degree of supercompactness of \(\kappa\) is sufficiently large, then \(\Ldthetasc_\kappa\)-trees are nothing new.

\begin{proposition}
	\label{prop:FullyTreeableThetaSCJLDExist}
	Suppose \(\kappa\) is \(\theta\)-supercompact and \(\theta\geq 2^\kappa\). Then
	a \(\Ldthetasc_\kappa\)-tree exists if and only if a 
	\(\theta\)-supercompactness Laver function for \(\kappa\) does
	(if and only if\/ \(\jLdthetasc_{\kappa,2^\kappa}\) holds).
\end{proposition}

\begin{proof}
	The forward implication is trivial, so we focus on the reverse implication. Let
	\(\ell\) be a Laver function for \(\kappa\). For any \(t\in\funcs{<\kappa}{2}\) define
	\(D(t)=\ell(|t|)(t)\) if this makes sense and \(D(t)=\emptyset\) otherwise. 
	We claim this defines a \(\Ldthetasc_\kappa\)-tree. Indeed, let 
	\(\vec{a}=\langle a_s;s\in\funcs{\kappa}{2}\rangle\) be a sequence of targets. Since
	\(\theta\geq 2^\kappa\) we get \(\vec{a}\in H_{\theta^+}\), so there is a 
	\(\theta\)-supercompactness embedding \(j\) such that \(j(\ell)(\kappa)=\vec{a}\).
	Therefore, given any \(s\in\funcs{\kappa}{2}\), we have
	\(
	j(D)(s)=j(\ell)(\kappa)(s)=a_s
	%\qedhere
	\)
\end{proof}

In other situations, however, the existence of a \(\Ldthetasc_\kappa\)-tree
can have strictly higher consistency strength than merely a 
\(\theta\)-supercompact cardinal.

\begin{definition}
	Let \(X\) be a set and \(\theta\) a cardinal.
	A cardinal \(\kappa\) is \emph{\(X\)-strong with closure \(\theta\)} if there is an
	elementary embedding \(j\colon V\to M\) with critical point \(\kappa\) such that
	\(\funcs{\theta}{M}\subseteq M\) and \(X\in M\).
\end{definition}

\begin{proposition}
	\label{prop:FullyTreeableThetaSCJLDStronger}
	Suppose \(\kappa\) is \(\theta\)-supercompact and there is a \(\Ldthetasc_\kappa\)-tree. 
	Then \(\kappa\) is \(X\)-strong with closure
	\(\theta\) for any \(X\subseteq H_{\theta^+}\) of size at most \(2^\kappa\).
\end{proposition}

\begin{proof}
	Suppose \(D\colon \funcs{<\kappa}{2}\to V_\kappa\) is a \(\Ldthetasc_\kappa\)-tree
	and fix an \(X\subseteq H_{\theta^+}\) of size at most \(2^\kappa\).
	Let \(f\colon \funcs{\kappa}{2}\to X\) enumerate \(X\). We can then find a 
	\(\theta\)-supercompactness embedding \(j\colon V\to M\) with critical point \(\kappa\)
	such that \(j(D)(s)=f(s)\) for all \(s\in\funcs{\kappa}{2}\). In particular,
	\(X=j(D)[\funcs{\kappa}{2}]\) is an element of \(M\), as required.
\end{proof}

%A kind of converse also holds. The proof of 
%proposition~\ref{prop:FullyTreeableThetaSCJLDExist}
%shows that if \(\kappa\) is \(X\)-strong with closure \(\theta\) for the \(X\) considered
%above and \(\kappa\) has a \(\theta\)-supercompactness Laver function then it also has a
%Laver tree.
If \(2^\kappa\leq\theta\) then \(X\)-strongness with closure \(\theta\) for
all \(X\subseteq H_{\theta^+}\) of size \(2^\kappa\) 
amounts to just \(\theta\)-supercompactness and 
Proposition~\ref{prop:FullyTreeableThetaSCJLDExist} gives the full equivalence of 
Laver functions
and \(\Ldthetasc_\kappa\)-trees. But if \(\theta< 2^\kappa\), then \(X\)-strongness with closure \(\theta\)
can have additional consistency strength. For example, 
we might choose \(X\) to be a normal measure on \(\kappa\) to see that \(\kappa\) must
have nontrivial Mitchell rank (by iterating this idea we can even deduce that
\(o(\kappa)=(2^\kappa)^+\)). In the typical scenario
where \(2^\kappa=2^\theta=\theta^+\), we can also reach higher and choose \(X\) to be a 
normal measure on \(\power_\kappa(\theta)\) and see that \(\kappa\) must also have 
nontrivial \(\theta\)-supercompactness Mitchell rank. 
We can use this observation to show
that there might not be any \(\Ldthetasc_\kappa\)-trees, even in the presence of very long joint Laver
sequences.

\begin{theorem}
	\label{thm:SeparateLongAndTreeableSCJLD}
	Suppose GCH holds and let \(\kappa\) be \(\theta\)-supercompact where either 
	\(\theta=\kappa\) or \(\cf(\theta)>\kappa\). Then there is a cardinal-preserving
	forcing extension in which \(\kappa\) remains \(\theta\)-supercompact, has a 
	\(\theta\)-supercompactness joint
	Laver sequence of length \(2^\kappa\), but is also the least measurable cardinal.
	In particular, \(\theta<2^\kappa\) and there are no \(\Ldthetasc_\kappa\)-trees in the extension.
\end{theorem}

\begin{proof}
	We may assume by prior forcing,	as in the proof of Theorem~\ref{thm:ForceLongJLDiamondSC}, 
	that \(\kappa\) has a Laver function. 
	Additionally, by performing either Magidor's iteration of Prikry forcing
	(see~\cite{Magidor1976:IdentityCrises}) or applying an argument due to
	Apter and Shelah (see~\cite{ApterShelah1997:MenasResultBestPossible}), depending
	on whether \(\theta=\kappa\) or not, we may assume that, in addition to being
	\(\theta\)-supercompact, \(\kappa\) is also the least measurable cardinal.\footnote{Of
		course, if \(\theta>\kappa\), this arrangement requires a strong failure of GCH at 
		\(\kappa\). In fact, \(2^\kappa=\theta^+\) in the Apter--Shelah model.}
	
	We now apply Corollary~\ref{cor:LeastSCCanHaveLongJLD} and
	arrive at a model where \(\kappa\) carries a \(\theta\)-su\-per\-com\-pact\-ness joint Laver sequence 
	of length \(2^\kappa\), but is also the least measurable cardinal. 
	It follows that there can be no 
	\(\Ldthetasc_\kappa\)-trees (or even \(\Ldmeas_\kappa\)-trees),
	since, by the discussion above, their existence would imply that \(\kappa\) has nontrivial Mitchell rank, implying that there are many measurables below \(\kappa\).
\end{proof}

Proposition~\ref{prop:FullyTreeableThetaSCJLDStronger} can be improved slightly to give a 
jump in consistency strength even for \ILdthetasc-trees where \(I\) is not the whole
set of branches.
A simple modification of the proof given there yields the following result, together
with the corresponding version of Theorem~\ref{thm:SeparateLongAndTreeableSCJLD}.

\begin{theorem}
	\label{thm:DefinableITreeableThetaSCJLDStronger}
	Suppose \(\kappa\) is \(\theta\)-supercompact and there is an \ILdthetasc-tree
	for some \(I\subseteq\funcs{\kappa}{2}\) of size \(2^\kappa\). 
	If \(I\) is definable (with parameters) over \(H_{\kappa^+}\)
	then \(\kappa\) is \(X\)-strong with closure \(\theta\) for any 
	\(X\subseteq H_{\theta^+}\) of size at most \(2^\kappa\).
\end{theorem}

The above theorem notwithstanding, 
we shall give a construction which shows that the 
existence of an \ILdthetasc-tree does not
yield additional consistency strength, provided that we allow \(I\) to be sufficiently
foreign to \(H_{\kappa^+}\). 
The argument will rely on being able to surgically alter
a Cohen subset of \(\kappa^+\) in a variety of ways. To this end we fix some
notation beforehand.

\begin{definition}
	Let \(f\) and \(g\) be functions.
	The \emph{graft} of \(f\) onto \(g\) is the function \(g\wr f\), defined on \(\dom(g)\) by
	\[(g\wr f)(x)=\begin{cases}f(x);&x\in\dom(g)\cap\dom(f)\\ 
	g(x);& x\in\dom(g)\setminus\dom(f)\end{cases}\]
\end{definition}

\noindent Essentially, the graft replaces the values of \(g\) with those of \(f\) on their 
common domain.

\begin{lemma}
	\label{lemma:ThinSubsets}
	Let \(\lambda\) be a regular cardinal and assume \(\diamond_\lambda\) holds.
	Suppose \(M\) is a transitive model of \textup{ZFC} (either set- or class-sized) such that
	\(\lambda\in M\) and \(M^{<\lambda}\subseteq M\) and \(|\mathcal{P}(\lambda)^M|=\lambda\).
	Then there are an unbounded  set \(I\subseteq\lambda\) and a function \(g\colon\lambda\to H_\lambda\)
	such that, given any \(f\colon I\to H_\lambda\), the graft \(g\wr f\) is generic for
	\(\Add(\lambda,1)\) over \(M\).\footnote{Here we take the version of
		\(\Add(\lambda,1)\) which adds a function \(g\colon\lambda\to H_\lambda\) by initial segments.}
\end{lemma}

\noindent The hypothesis of \(\diamond_\lambda\) is often automatically satisfied.
Specifically, our assumptions about \(M\) imply that \(2^{<\lambda}=\lambda\).
If \(\lambda=\kappa^+\) is a successor, this gives
\(2^\kappa=\kappa^+\) which already implies \(\diamond_\lambda\) by
a result of Shelah~\cite{Shelah2010:Diamonds}.

We are grateful to Joel David Hamkins for suggesting this proof.

\begin{proof}
	Let \(\langle f_\alpha;\alpha<\lambda\rangle\), with 
	\(f_\alpha\colon \alpha\to H_{|\alpha|}\), be a \(\diamond_\lambda\)-sequence and
	fix an enumeration \(\langle D_\alpha;\alpha<\lambda\rangle\) of the open dense subsets of
	\(\Add(\lambda,1)\) in \(M\). We shall construct by recursion a descending sequence of
	conditions \(p_\alpha\in\Add(\lambda,1)\) and an increasing sequence of sets
	\(I_\alpha\) as approximations to \(g\) and \(I\). Specifically, we shall use
	\(\diamond_\lambda\) to guess pieces of any potential function \(f\) and ensure along the
	way that the modified conditions \(p_\alpha\wr f\) meet all of the listed dense sets.
	
	Suppose we have built the sequences \(\langle p_\alpha;\alpha<\gamma\rangle\) and
	\(\langle I_\alpha;\alpha<\gamma\rangle\) for some \(\gamma<\lambda\).
	Let \(I_\gamma^*=\bigcup_{\alpha<\gamma} I_\alpha\).
	Let \(p_\gamma^*\in M\) be an extension of \(\bigcup_{\alpha<\gamma} p_\alpha\) 
	such that \(I_\gamma^*\subseteq\dom(p_\gamma^*)\in\lambda\); such an extension exists
	in \(M\) by our assumption on the closure of \(M\).
	
	Let us briefly summarize the construction. We shall surgically modify the condition 
	\(p_\gamma^*\) by grafting the function given by \(\diamond_\lambda\) onto it.
	We shall then extend this modified condition to meet one of our dense sets, after
	which we will undo the surgery. We will be left with a condition \(p_\gamma\)
	which is one step closer to ensuring that the result of one particular grafting
	\(g\wr f\) is generic. At the same time we also extend \(I_\gamma^*\) by
	adding a point beyond the domains of all the conditions constructed so far.
	
	More precisely, let  \(\tilde{p}^*_\gamma=p^*_\gamma \wr (f_\gamma\rest I^*_\gamma)\).
	This is still a condition in \(M\).
	Let \(\widetilde{p}_\gamma\) be any extension of this condition
	inside \(D_{\eta_\gamma}\), where \(\eta_\gamma\) is the least such that
	\(\tilde{p}^*_\gamma\notin D_{\eta_\gamma}\), and satisfying 
	\(\dom(\widetilde{p}_\gamma)\in\lambda\). Finally, we undo the initial graft and set
	\(p_\gamma=\widetilde{p}_\gamma \wr (p^*_\gamma\rest I^*_\gamma)\). Note that
	we have \(p_\gamma\leq p^*_\gamma\). We also extend our approximation to \(I\) with
	the first available point, letting
	\(I_\gamma=I_\gamma^*\cup\{\min(\lambda\setminus\dom(p_\gamma))\}\).
	
	Once we have completed this recursive construction we can set
	\(I=\bigcup_{\gamma<\lambda}I_\gamma\) and \(g=\bigcup_{\gamma<\lambda}p_\gamma\).
	Let us check that these do in fact have the desired properties.
	
	Let \(f\colon I\to H_\lambda\) be a function. We need to show that \(g\wr f\) is generic over
	\(M\). Using \(\diamond_\lambda\), we find that there are stationarily many \(\gamma\)
	such that \(f_\gamma=f\rest\gamma\).
	Note also that there are club many \(\gamma\) such that \(I_\gamma^*\subseteq\gamma\)
	is unbounded, and together this means that 
	\(S=\set{\gamma}{f_\gamma\rest I_\gamma^*=f\rest(I\cap\gamma)}\) is stationary.
	The conditions \(\widetilde{p}_\gamma\) for \(\gamma\in S\) extend each other
	and we have \(\bigcup_{\gamma\in S}\widetilde{p}_\gamma=g\wr f\). Furthermore, since
	the sets \(D_\alpha\) are open, the construction of \(\widetilde{p}_\gamma\) ensures
	that eventually these conditions will meet all of these dense sets, showing that \(g\wr f\)
	really is generic.
\end{proof}

The construction in the above proof is quite flexible and can be modified to make the
set \(I\) generic in various ways as well (for example, we can arrange for \(I\) to be
Cohen or dominating over \(M\), and have other similar properties).

\begin{theorem}
	\label{thm:ForceITreeableThetaSCJLD}
	If \(\kappa\) is \(\theta\)-supercompact, then there is a forcing extension in which there is
	an \ILdthetasc-tree for some \(I\subseteq\funcs{\kappa}{2}\) of size \(2^\kappa\).
\end{theorem}

\begin{proof}
	If \(\theta\geq 2^\kappa\) then even a single Laver function for \(\kappa\) gives rise
	to a full \(\Ld_\kappa\)-tree, by Proposition~\ref{prop:FullyTreeableThetaSCJLDExist},
	and we can force the existence of a Laver function by Theorem~\ref{thm:ForceLongJLDiamondSC}.
	We thus focus on the remaining case when \(\kappa\leq\theta<2^\kappa\).
	
	We make similar simplifying assumptions as in Theorem~\ref{thm:ForceLongJLDiamondSC}.
	Just as there we assume that \(\theta=\theta^{<\kappa}\).
	Furthermore, we may assume that \(2^\theta=\theta^+\), since this can be forced
	without 
	%adding subsets to \(\power_\kappa(\theta)\) and 
	affecting the \(\theta\)-supercompactness of \(\kappa\). Note that these cardinal arithmetic hypotheses
	imply that \(2^\kappa=\theta^+\).
	
	Let \(\P\) be the length \(\kappa\) Easton support iteration
	which adds, in a recursive fashion, a labelling of the tree \(\funcs{<\kappa}{2}\)
	of the extension. Specifically, let \(\P\)
	force with \(\Q_\gamma=\Add(2^\gamma,1)\) at each inaccessible \(\gamma<\kappa\) stage
	\(\gamma\). Let \(G\subseteq\P\) be generic and let \(G_\gamma\) be the piece added at stage 
	\(\gamma\).
	Using suitable coding, we can see each \(G_\gamma\), in \(V[G]\),
	as a function \(G_\gamma\colon \funcs{\gamma}{2}\to H_{\gamma^+}\);
	in particular, \(G_\gamma\)	really does label the whole level \(\funcs{\gamma}{2}\) in the
	final extension, since no new nodes appear in the tree \(\funcs{\leq\gamma}{2}\) after stage
	\(\gamma\) in the iteration \(\P\).
	Thus \(G\) induces a map \(D\colon \funcs{<\kappa}{2}\to V_\kappa[G]\), by 
	extending the \(G_\gamma\) in any way we like to the entire tree. 
	We shall show that \(D\) is an \ILdthetasc-tree for some specifically chosen \(I\).
	
	Fix a
	\(\theta\)-supercompactness embedding \(j\colon V\to M\) in \(V\). Note that
	\(M[G]^{\theta}\subseteq M[G]\) in \(V[G]\) as well, since the forcing \(\P\) is
	\(\theta^+\)-cc. Furthermore, in \(V[G]\) we still have \(2^\theta=\theta^+\), which
	implies \(\diamond_{\theta^+}\) by a result of Shelah~\cite{Shelah2010:Diamonds}.
	We also know that \(|\power(\theta^+)^{M[G]}|=\theta^+\), since \(|j(\kappa)|=\theta^+\) and
	\(|\power(\theta^+)^{M[G]}|<j(\kappa)\) because \(\theta<j(\kappa)\) and \(j(\kappa)\) is
	inaccessible in \(M[G]\).
	Now apply Lemma~\ref{lemma:ThinSubsets} to \(M[G]\) and \(\lambda=\theta^+\) to obtain
	an \(I\subseteq \funcs{\kappa}{2}\) of size \(\theta^+\) and a function 
	\(g\colon \funcs{\kappa}{2}\to H_{\theta^+}\)
	such that for any \(f\colon I\to H_{\theta^+}\) in \(V[G]\), 
	the graft \(g\wr f\) is generic over
	\(M[G]\). We claim that \(D\) is an \ILdthetasc-tree.
	
	To check the guessing property, fix a sequence of targets 
	\(\vec{a}=\langle a_s;s\in I\rangle\) in \(V[G]\). We shall lift the embedding \(j\)
	to \(V[G]\). Let us write \(j(\P)=\P*\Q_\kappa*\Ptail\).
	We know that \(g\wr \vec{a}\) is \(M[G]\)-generic for \(\Q_\kappa\), so we only need to find
	the further generic for \(\Ptail\). We easily see that
	\(M[G][g\wr \vec{a}]^{\theta}\subseteq M[G][g\wr \vec{a}]\) in \(V[G]\), that \(\Ptail\) is
	\(\leq\theta\)-closed in that model, and that \(M[G][g\wr \vec{a}]\) only has
	\(\theta^+\)-many subsets of \(\Ptail\). We can thus diagonalize against these
	dense sets in \(\theta^+\)-many steps and produce a generic \(\Gtail\) for
	\(\Ptail\). Putting all of this together, we can lift \(j\) to
	\(j\colon V[G]\to M[j(G)]\) in \(V[G]\), where \(j(G)=G*(g\wr \vec{a})*\Gtail\). 
	Now consider \(j(D)\). This is
	exactly the labelling of the tree \(\funcs{<j(\kappa)}{2}\) in \(M[j(G)]\)
	given by \(j(G)\). Furthermore, for any \(s\in I\), we have 
	\(j(D)(s)=(g\wr \vec{a})(s)=a_s\), verifying the guessing property.
\end{proof}

Given a \(\Ld_\kappa\)-tree, it is easy to produce a joint Laver sequence of length
\(2^\kappa\) from it by just reading the labels along each branch of the tree.
The resulting sequence then exhibits a large degree of coherence.
We might wonder about the possibility of reversing this process, starting with a joint
Laver sequence and attempting to fit it into a tree. But, taken literally, this property is not very robust.
For example, all functions in such a joint Laver sequence must have the same value at
0, so this coherence property of joint Laver sequences can be destroyed without changing
the sequence in any essential way. To avoid such trivialities, we relax the definition to only
ask for coherence modulo bounded perturbations.

\begin{definition}
	\label{def:treeability}
	Let \(\kappa\) be a regular cardinal and \(\vec{f}=\langle f_\alpha;\alpha<\lambda\rangle\)
	a sequence of functions defined on \(\kappa\). The sequence \(\vec{f}\) is
	\emph{treeable} if there are a bijection \(e\colon \lambda\to \funcs{\kappa}{2}\)
	and a tree labelling \(D\colon \funcs{<\kappa}{2}\to V\) such that, for all \(\alpha<\lambda\),
	we have \(D(e(\alpha)\rest\xi)=f_\alpha(\xi)\) for all but boundedly many \(\xi<\kappa\).
%	
%	Given an \(I\subseteq\funcs{\kappa}{2}\), the sequence \(\vec{f}\) is \emph{\(I\)-treeable}
%	if the above holds for a bijection \(e\colon \lambda\to I\).
\end{definition}

\begin{lemma}
	\label{lemma:GenericNotTreeable}
	Let \(\kappa\) be a regular cardinal and assume that \(2^{<\kappa}< 2^\kappa\). 
	Let \(G\subseteq\Add(\kappa,2^\kappa)\) be generic. Then \(G\) is not treeable.
\end{lemma}

\begin{proof}
	Let us write \(G=\langle g_\alpha;\alpha<2^\kappa\rangle\) as a sequence of its slices.
	Now suppose that this
	sequence were treeable and let \(\dot{e}\) and \(\dot{D}\) be names for the indexing function
	and the labelling of \(\funcs{<\kappa}{2}\), respectively. Our cardinal arithmetic
	assumption implies that the name \(\dot{D}\) only
	involves conditions from a bounded part of the poset \(\Add(\kappa,2^\kappa)\), so we
	may assume that the labelling \(D\) exists already in the ground model.
	Let \(p\) be an arbitrary condition and \(\alpha<\kappa\). Since we assumed that
	\(G\) was forced to be treeable, there is a name \(\dot{\gamma}\)
	for an ordinal beyond which \(G_\alpha\) agrees with \(D\rest f(\alpha)\).
	By strengthening \(p\) if necessary, we may assume that the value of \(\dot{\gamma}\) has
	been decided. We now inductively construct a countable descending sequence of conditions 
	below \(p\), deciding longer and longer initial segments of \(\dot{e}(\alpha)\),
	in such a way that, for some \(\delta>\gamma\), their union \(p^*\leq p\) decides
	\(\dot{e}(\alpha)\rest\delta\) but does not decide \(G_\alpha(\delta)\).
	Then \(p^*\) can be further extended to a condition forcing
	\(G_\alpha(\delta)\neq D(\dot{e}(\alpha)\rest\delta)\), which contradicts the fact
	that \(p\) forces that \(G\) is treeable.
\end{proof}

\begin{corollary}
	\label{cor:NontreeableSCJLD}
	If \(\kappa\) is \(\theta\)-supercompact, then there is a forcing extension in which
	there is a nontreeable \(\theta\)-supercompactness joint Laver sequence for \(\kappa\) of 
	length \(2^\kappa\).
\end{corollary}

\begin{proof}
	The joint Laver sequence constructed in Theorem~\ref{thm:ForceLongJLDiamondSC} was
	added by forcing with \(\Add(\kappa,2^\kappa)\), so it is not treeable by
	Lemma~\ref{lemma:GenericNotTreeable}.
\end{proof}

\section{Joint Laver diamonds for strong cardinals}
\label{sec:JLDStrong}
\begin{definition}
	%Let \(\kappa\) be \(\theta\)-strong. 
	A function \(\ell\colon\kappa\to V_\kappa\) is
	a \(\theta\)-strongness Laver function if it guesses elements of \(V_\theta\)
	via \(\theta\)-strongness embeddings with critical point \(\kappa\).
	
	If \(\kappa\) is fully strong then a function \(\ell\colon\kappa\to V_\kappa\) is a
	Laver function for \(\kappa\) if it is a \(\theta\)-strongness Laver function
	for \(\kappa\) for all \(\theta\).
\end{definition}

Just as in the supercompact case, we shall say that \(\Ldthetastr_\kappa\) holds if there
is a \(\theta\)-strongness Laver function for \(\kappa\); since our default set of targets
in this case is \(V_\theta\), this is just the same as \(\Ldthetastr_\kappa(V_\theta)\).
In the case of full strongness, we shall similarly say that \(\Ldstr_\kappa\) holds if there
is a strongness Laver function for \(\kappa\), which should be read as \(\Ldstr_\kappa(V)\).

A similar factoring argument as in the supercompact case shows that we can afford to be
imprecise about which embeddings we count as \(\theta\)-strongness embeddings in the definition above.
Specifically, if \(j\colon V\to M\) is any \(\theta\)-strongness embedding with critical point
\(\kappa\) and a function \(\ell\) guesses a target \(a\in V_\theta\) via \(j\), then
\(\ell\) also guesses \(a\) via the induced \((\kappa,V_\theta)\)-extender ultrapower embedding.

As in the supercompact case, \(2^\kappa\) is the largest possible cardinal length of a
\(\theta\)-strongness joint Laver sequence for \(\kappa\), just because there are only
\(2^\kappa\) many functions \(\ell\colon\kappa\to V_\kappa\).

The set of targets \(V_\theta\) is a bit unwieldy and lacks some basic closure
properties, particularly in the case when \(\theta\) is a successor ordinal.
The following lemma shows that, modulo some coding, we can recover a good deal of closure
under sequences.

\begin{lemma}
	\label{lemma:FlatCoding}
	Let \(\theta\) be an infinite ordinal and let \(I\in V_\theta\) be a set.
	If \(\theta\) is  successor ordinal or \(\cf(\theta)>|I|\) then \(V_\theta\) is closed
	under a coding scheme for sequences indexed by \(I\). Moreover, this coding is
	\(\Delta_0\)-definable.
\end{lemma}

\begin{proof}
	If \(\theta=\omega\) then the \(I\) under consideration are finite. Since \(V_\omega\)
	is already closed under finite sequences we need only deal with \(\theta>\omega\).
	
	Fix in advance a simply definable flat pairing function \([\cdot,\cdot]\)
	(flat in the sense that any infinite \(V_\alpha\) is closed under it; the Quine--Rosser
	pairing function will do).
	
	Let \(\vec{a}=\langle a_i; i\in I\rangle\) be a sequence of elements of \(V_\theta\).
	For each \(i\in I\) we can find an (infinite) ordinal \(\theta_i<\theta\) such that
	\(a_i\cup \{i\}\subseteq V_{\theta_i}\). Now let \(\widetilde{a}_i=\set{[i,b]}{b\in a_i}
	\subseteq V_{\theta_i}\) and finally define
	\(\widetilde{a}=\bigcup_{i\in I}\widetilde{a}_i\). We see that 
	\(\widetilde{a}\subseteq V_{\sup_i \theta_i}\) and, under our hypotheses,
	\(\sup_i\theta_i<\theta\). It follows that
	\(\widetilde{a}\in V_{(\sup_i\theta_i)+1} \subseteq V_\theta\) as required.
\end{proof}

\begin{proposition}
	\label{prop:ThetaStrJLDiamond}
	Let \(\kappa\) be \(\theta\)-strong with \(\kappa+2\leq\theta\) and let 
	\(\lambda\leq 2^\kappa\) be a cardinal. If there is a \(\theta\)-strongness Laver 
	function for \(\kappa\)
	and \(\theta\) is either a successor ordinal or \(\lambda<\cf(\theta)\) then there is a
	\(\theta\)-strongness joint Laver sequence of length \(\lambda\) for \(\kappa\).
\end{proposition}

\noindent In particular, if \(\theta\) is a successor, then a single \(\theta\)-strongness
Laver function already yields a joint Laver sequence of length \(2^\kappa\), the maximal
possible.

\begin{proof}
	We aim to imitate the proof of Proposition~\ref{prop:SCHasSomeJLDiamond}. 
	To that end, fix an \(I\subseteq\mathcal{P}(\kappa)\) of size \(\lambda\) and a
	bijection \(f\colon \lambda\to I\).
	If \(\ell\) is a Laver function for \(\kappa\), we define a 
	joint Laver sequence by letting \(\ell_\alpha(\xi)\), for each \(\alpha<\lambda\),
	be the element of \(\ell(\xi)\) with index \(f(\alpha)\cap \xi\) in the coding scheme
	described in Lemma~\ref{lemma:FlatCoding}.
	
	It is now easy to verify that the functions \(\ell_\alpha\) form a joint Laver sequence:
	given a sequence of targets \(\vec{a}\), we can replace it with a coded version \(\widetilde{a}\in V_\theta\), by using \(f\) and 
	Lemma~\ref{lemma:FlatCoding}.
	We can then use \(\ell\) to guess \(\widetilde{a}\) and the \(\theta\)-strongness
	embedding obtained this way will witness the joint guessing property of the \(\ell_\alpha\).
\end{proof}

Again, as in the supercompact case, if the Laver diamond we started with works for
several different \(\theta\) then the joint Laver sequence derived above
will also work for those same \(\theta\). In particular, if \(\kappa\) is strong,
then combining the argument from Proposition~\ref{prop:ThetaStrJLDiamond} with
the construction of a strongness Laver function due to Gitik and Shelah~\cite{GitikShelah1989:IndestructibilityOfStrong} gives an analogue of Corollary~\ref{cor:SCHasLongJLDiamond}.

\begin{corollary}
	\label{cor:StrHasJLD}
	If \(\kappa\) is strong then there is a strongness joint Laver sequence for \(\kappa\)
	of length \(2^\kappa\).
\end{corollary}

Proposition~\ref{prop:ThetaStrJLDiamond} implies that in most cases (that is, for most 
\(\theta\)) we do not need to do any additional work beyond ensuring that there is a 
\(\theta\)-strongness Laver function for \(\kappa\) to automatically also find the longest 
possible joint Laver sequence. For example, if
\(\theta\) is a successor or if \(\cf(\theta)>2^\kappa\) then a single \(\theta\)-strongness
Laver function yields a joint Laver sequence of length \(2^\kappa\).
To gauge the consistency strength of the existence of \(\theta\)-strongness
joint Laver sequences for \(\kappa\) we should therefore only focus on the consistency
strength required for a single Laver diamond, and, separately, 
on \(\theta\) of low cofinality.

%Forcing constructions giving a single \(\theta\)-strongness Laver diamond for a cardinal
%\(\kappa\) are not difficult (one can simply add, after a suitable preparation, a Cohen subset
%to \(\kappa\); a simple master condition argument shows that this adds a Laver function).
%We therefore get an immediate corollary to Proposition~\ref{prop:ThetaStrJLDiamond}.
%%%%%%%%%%%%%%%%%%%%%%%%%%%%%

Forcing constructions for a single \(\theta\)-strongness Laver diamond are known. However, since
we weren't able to find a suitable reference, we give the proofs in some detail.

We are going to need an analogue of Lemma~\ref{lemma:ThetaSCMenasFunction} for strong cardinals.

\begin{definition}
	A \(\theta\)-strongness \emph{Menas function} for a cardinal \(\kappa\) is a function
	\(f\colon\kappa\to\kappa\) such that there is a \(\theta\)-strongness embedding
	\(j\colon V\to M\) with \(\cp(j)=\kappa\) and \(j(f)(\kappa)=\theta\).
\end{definition}

We should mention that our definition differs slightly from Hamkins' original definition 
in~\cite{Hamkins2000:LotteryPreparation}, where he says that \(f\) is a \(\theta\)-strongness
Menas function if \(j(f)(\kappa)>\beth_\theta\) for some \(\theta\)-strongness embedding \(j\).
A Menas function in our sense gives rise to one in Hamkins' original sense, and
is in general more convenient to work with.

\begin{lemma}
	\label{lemma:thetaStrMenasFunction}
	If \(\kappa\) is \(\theta\)-strong, then \(\kappa\) carries a \(\theta\)-strongness
	Menas function.
\end{lemma}

\begin{proof}
	We define a function \(f\colon \kappa\to\kappa\) by letting \(f(\alpha)=0\) if \(\alpha\)
	is \(\kappa\)-strong, and otherwise let \(f(\alpha)\) be the least \(\gamma<\kappa\) such that
	\(\alpha\) is not \(\gamma\)-strong. An argument much like the one in 
	Lemma~\ref{lemma:ThetaSCMenasFunction} shows that we can find a \(\theta\)-strongness
	embedding \(j\colon V\to M\) with critical point \(\kappa\) such that \(\kappa\) is not
	\(\theta\)-strong in \(M\), and that \(j(f)(\kappa)=\theta\) for this \(j\).
\end{proof}

%The following two lemmas are implicit in many existing arguments involving the preservation of
%strongness; we include the proofs for completeness.

\begin{lemma}
	\label{lemma:ValphaNamesSmallForcing}
	Let \(\kappa\) be a cardinal and suppose \(\P\subseteq V_\kappa\) is a
	poset. Let \(G\subseteq \P\) be generic and assume that \(\kappa\) is a \(\beth\)-fixed
	point in \(V[G]\). For any ordinal \(\alpha\geq\kappa+1\),
	every element of \(V[G]_\alpha=(V_\alpha)^{V[G]}\) has a \(\P\)-name coded in \(V_\alpha\).
\end{lemma}

\begin{proof}
	We argue by induction on \(\alpha\). In the base step, the key point is that, since
	\(\kappa\) is a \(\beth\)-fixed point in \(V[G]\), names for elements of \(V[G]_{\kappa+1}\)
	can be replaced with names for subsets of \(\kappa\). But nice names for subsets of
	\(\kappa\) are essentially just subsets of \(\kappa\times\P\) and are thus elements of
	\(V_{\kappa+1}\). The limit step of the induction is trivial, so it only remains for
	us to consider the successor step.
	
	Assume that every element of \(V[G]_\alpha\), for some \(\alpha\geq \kappa+1\), has a name
	coded in \(V_\alpha\). Now consider an arbitrary element of \(V[G]_{\alpha+1}\).
	By the induction hypothesis, it has a nice name of the form
	\(\sigma=\bigcup_x \{x\}\times A_x\), where each \(A_x\) is an antichain in \(\P\)
	and the union runs over all the coded names \(x\) in \(V_\alpha\). We can think of
	\(\sigma\) as simply the sequence of the antichains \(A_x\), indexed by a subset of \(V_\alpha\),
	and, because each \(A_x\) is an element of \(V_{\kappa+1}\subseteq V_{\alpha}\),
	Lemma~\ref{lemma:FlatCoding} implies that this sequence is coded by an element of 
	\(V_{\alpha+1}\).
%		
%	The proof is quite similar to the one we just gave and we can argue by induction on \(\alpha\) again. The limit case is clear, and the successor step goes exactly like the successor
%	step in the previous proof, using Lemma~\ref{lemma:FlatCoding}. 
%	For the base case \(\alpha=\kappa+1\), the key point is that,
%	since \(\kappa\) is a \(\beth\)-fixed point in \(V[G]\), 
%	names for elements of \(V[G]_{\kappa+1}\) can be recovered from names for subsets of 
%	\(\kappa\). But nice names for subsets of \(\kappa\) are essentially just subsets
%	of \(\kappa\times \P\) and are thus elements of \(V_{\kappa+1}\).
\end{proof}

\begin{lemma}
	\label{lemma:ValphaNamesAddKappa+}
	Let \(\kappa\) be a cardinal and let \(\P=\Add(\kappa^+,1)\). Let \(G\subseteq \P\) be generic.
	For any ordinal \(\alpha\geq\kappa\),
	every element of \(V[G]_\alpha=(V_\alpha)^{V[G]}\) has a \(\P\)-name coded in \(V_\alpha\). 
\end{lemma}

\begin{proof}
	We can argue much like in the proof we just gave, by induction on \(\alpha\). 
	For \(\alpha=\kappa\) or \(\alpha=\kappa+1\),
	no new elements of \(V_\alpha\) are added by \(\P\), so we can simply take check names.
	The limit step of the induction is trivial, so it only remains for us to consider the
	successor step.
	
	Assume that every element of \(V[G]_\alpha\), for some \(\alpha\geq\kappa+2\), has
	a name coded in \(V_\alpha\). We may as well work with an isomorphic copy of the poset \(\P\)
	which is a subset of \(V_{\kappa+1}\). This means that every antichain of \(\P\) is an element
	of \(V_{\kappa+2}\subseteq V_\alpha\). Now consider an arbitrary element of \(V[G]_{\alpha+1}\).
	By the induction hypothesis, it has a nice name of the form \(\sigma=\bigcup_x \{x\}\times A_x\), 
	where each \(A_x\) is an antichain in \(\P\), and the union runs over all the coded names in 
	\(V_\alpha\) for elements of \((V[G])_\alpha\). As in the previous proof, we think of \(\sigma\)
	as a sequence of antichains indexed by a subset of \(V_\alpha\), and use 
	Lemma~\ref{lemma:FlatCoding} to obtain a code for \(\sigma\) in \(V_{\alpha+1}\).
\end{proof}

\begin{theorem}
	\label{thm:forceThetaStrLd}
	If \(\kappa\) is \(\theta\)-strong and \(\theta\) is either a successor ordinal or
	\(\cf(\theta)\geq\kappa^+\), then there is a \((2^\kappa)^+\)-cc forcing extension in which
	there is a \(\theta\)-strongness Laver function for \(\kappa\). In the extension 
	\(2^\kappa=\kappa^+\) holds, \(\kappa^+\) is preserved, and, 
	if \(\theta\) was a limit ordinal, \(\cf(\theta)\geq\kappa^+\) remains true.
\end{theorem}

\begin{proof}
	If \(\theta\leq\kappa\), no forcing is necessary, and the case \(\theta=\kappa+1\)
	was essentially covered by Theorem~\ref{thm:ForceLongJLDiamondSC}. We may therefore
	restrict our attention to the case when \(\theta\geq\kappa+2\).
	
	Fix a Menas function \(f\) as in Lemma~\ref{lemma:thetaStrMenasFunction} and define
	the forcing \(\P_\kappa\) as the length \(\kappa\) Easton support iteration which forces
	with \(\Q_\gamma=\Add(\gamma^+,1)\) at inaccessible closure points \(\gamma\)
	of the function \(f\). Let \(\P=\P_\kappa * \Q_\kappa\) and let 
	\(G=G_\kappa* g\subseteq \P\) be generic over \(V\). Let us first show that \(\kappa\)
	remains \(\theta\)-strong in \(V[G]\). We will describe later how to derive a Laver function
	in the extension.
	
	By the definition of the Menas function \(f\), we can fix a \(\theta\)-strongness extender 
	embedding \(j\colon V\to M\) such that \(j(f)(\kappa)=\theta\). 
	We can write \(j(\P)=\P*\Ptail*j(\Q_\kappa)\).

	Because of our assumptions about \(\theta\), Lemma~\ref{lemma:FlatCoding} implies
	that \(M\) is in fact closed under \(\kappa\)-sequences.	
	Since \(\P_\kappa\) is \(\kappa\)-cc, the model \(M[G_\kappa]\) remains closed under
	\(\kappa\)-sequences in \(V[G_\kappa]\), and, since \(\Q_\kappa\) is \(\leq\kappa\)-closed,
	this remains true for \(M[G]\) in \(V[G]\).
	
	Recall that every element of \(M[G]\) is of the form \(j(F)(a)^G\) for some \(a\in V_\theta\)
	and some function \(F\in V\) defined on \(V_\kappa\). Since \(\Ptail\) has size \(j(\kappa)\)
	in \(M[G]\), every open dense subset of \(\Ptail\) can be represented in this way by using
	a function \(F\colon V_\kappa\to V_{\kappa+1}\). For a fixed \(F\) like this, let
	\(\mathcal{D}_F\) be the collection of all open dense subsets of \(\Ptail\) of the form
	\(j(F)(a)^G\) for some \(a\in V_\theta\). Then \(\mathcal{D}_F\in M[G]\) and it
	has size \(\beth_\theta\) in \(M[G]\). Since the first stage of forcing in \(\Ptail\) occurs
	after the first inaccessible above \(j(f)(\kappa)=\theta\) in \(M\), the forcing \(\Ptail\)
	is \(\leq\beth_\theta\)-closed in \(M[G]\). This means that \(D_F=\bigcap\mathcal{D}_F\)
	is a dense subset of \(\Ptail\). Since \(2^\kappa=\kappa^+\) in \(V[G]\) (because of the last
	stage of forcing), there are only \(\kappa^+\) many of these dense sets \(D_F\), counted in
	\(V[G]\). Therefore we can line them up and, using the closure of the poset \(\Ptail\)
	and of the model \(M[G]\), meet all of them in order to build a generic \(\Gtail\in V[G]\)
	over \(M[G]\). This allows us to lift the embedding \(j\) to
	\(j\colon V[G_\kappa]\to M[G][\Gtail]\).
	
	The final lift is easier: since \(\Q_\kappa\) is \(\leq\kappa\)-distributive,
	the filter \(h\) generated by the image \(j[g]\) is generic over \(M[G][\Gtail]\) and
	allows us to lift to \(j\colon V[G]\to M[G][\Gtail][h]\). To see that this embedding
	witnesses the \(\theta\)-strongness of \(\kappa\) in \(V[G]\), we use 
	Lemmas~\ref{lemma:ValphaNamesSmallForcing} and~\ref{lemma:ValphaNamesAddKappa+}. 
	Lemma~\ref{lemma:ValphaNamesSmallForcing} implies that
	every element of \(V[G_\kappa]_\theta\) has a name coded in \(V_\theta\), so, since
	\(V_\theta\in M\) and \(G_\kappa\in M[G]\), we get \(V[G_\kappa]_\theta\in 
	M[G]\subseteq M[G][\Gtail][h]\). 
	Lemma~\ref{lemma:ValphaNamesAddKappa+}, in turn, says that every element of 
	\(V[G]_\theta\)	has a name coded in \(V[G_\kappa]_\theta\), and so, because 
	\(g\in M[G]\), we get \(V[G]_\theta\in M[G]\subseteq M[G][\Gtail][h]\).
	This shows that \(\kappa\) is \(\theta\)-strong in \(V[G]\).
%
%	Let \(i\colon V\to N\) be the ultrapower
%	by the induced normal measure on \(\kappa\) and let \(k\colon N\to M\) be the factor embedding.
%	
%	\[\begin{tikzcd}
%	V \arrow[r, "j"] \arrow[rd, "i", swap] & M\\ & N \arrow[u, "k", swap]
%	\end{tikzcd}\]
%		
%	\noindent We can write \(j(\P)=\P*\Ptail*j(\Q_\kappa)\) and 
%	\(i(\P)=\P*\Ptail'* i(\Q_\kappa)\). We shall first lift the embedding \(i\)
%	and use that to help us lift \(j\).
%	
%	Since \(\P_\kappa\) is \(\kappa\)-cc, the model \(N[G_\kappa]\) remains closed under
%	\(\kappa\)-sequences in \(V[G_\kappa]\), and, since \(\Q_\kappa\) is \(\leq\kappa\)-closed,
%	this remains true for \(N[G]\) in \(V[G]\).
%	The forcing \(\Ptail'\) has size \(i(\kappa)\) and is \(\leq\kappa\)-closed
%	in \(N[G]\). Typically at this stage we would use some kind of GCH hypothesis in \(V\)
%	to bound the number of dense subsets of \(\Ptail'\) in \(N[G]\). We haven't made any such
%	assumption, however, because of the forcing \(\Q_\kappa\), we \emph{do} have \(2^\kappa=\kappa^+\)
%	in \(V[G]\). This is enough to see that, counted in \(V[G]\), there are only \(\kappa^+\)
%	many dense subsets of \(\Ptail'\) in \(N[G]\). Using the closure of \(N[G]\) and \(\Ptail'\),
%	we can thus find a generic \(\Gtail'\) over \(N[G]\), and lift \(i\) to
%	\(i\colon V[G_\kappa]\to N[G][\Gtail']\). Since the final forcing \(\Q_\kappa\) is
%	\(\leq\kappa\)-closed in \(V[G_\kappa]\), the image \(i[g]\) generates a generic filter
%	\(h'\) over \(N[G][\Gtail']\), which allows us to completely lift \(i\) to
%	\(i\colon V[G]\to N[G][\Gtail'][h']\).

Now let us turn our attention to the Laver function. We define a function \(\ell\in V[G]\) on
\(\kappa\) by letting \(\ell(\xi)\) be the set whose Mostowski code appears as the first set
coded in the sequence of bits \(G_{\xi'}\), where \(\xi'\) is the least inaccessible closure
point of \(f\) above \(\xi\). It follows from this definition that, given a \(\theta\)-strongness
embedding \(j\) in \(V[G]\), we evaluate \(j(\ell)(\kappa)\) by consulting the first slice
of the generic \(\Gtail\), as constructed above, and seeing what is coded there. To show
that \(\ell\) really is a Laver function, it only remains for us to see that any target
\(a\in V[G]_\theta\) can be coded into \(\Gtail\) appropriately. But this is straightforward:
if we start with a \(\theta\)-strongness embedding in \(V\) as above and proceed with the
lifting argument, we saw that \(V[G]_\theta\subseteq M[G]\), so the target \(a\) appears in
\(M[G]\). Consequently, the bit sequence of the Mostowski code of \(a\) is a condition
in the first stage of forcing in \(\Ptail\). If we now run our construction of \(\Gtail\)
as described above, with the added requirement that we start with the fixed condition coding \(a\)
before we meet all of the dense sets \(D_F\), we will obtain a generic \(\Gtail\) whose first
slice codes exactly \(a\), and a lifted embedding \(j\) satisfying \(j(\ell)(\kappa)=a\).
Therefore \(\ell\) really is a \(\theta\)-strongness Laver function in \(V[G]\).
\end{proof}

Combining Theorem~\ref{thm:forceThetaStrLd} with Proposition~\ref{prop:ThetaStrJLDiamond},
we immediately obtain the following corollary.

\begin{corollary}
	Let \(\kappa\) be \(\theta\)-strong with \(\kappa+2\leq\theta\). 
	If \(\theta\) is either a successor ordinal or \(\cf(\theta)>2^\kappa\) then
	there is a forcing extension in which \(\jLdthetastr_{\kappa,2^\kappa}\) holds.
\end{corollary}

Again, the forcing we do in Theorem~\ref{thm:forceThetaStrLd} will collapse \(2^\kappa\)
to \(\kappa^+\), so the conclusion \(\jLdthetastr_{\kappa,2^\kappa}\) in this corollary 
should be read with the parameters evaluated in the extension.

%
%It is worth noting that any effort of adapting the preceding proof in the style of
%theorem~\ref{thm:ForceLongJLDiamondSC} to give joint Laver sequences seems to require
%exactly the hypotheses of proposition~\ref{prop:ThetaStrJLDiamond}. We have therefore
%decided to give two separate, easier arguments, instead of a single intertwined one.
Moving on to the case of \(\theta\) of low cofinality, 
it is important to note that Theorem~\ref{thm:forceThetaStrLd} has little to say in this
situation. In fact we do not know whether one can force the existence of a
\(\theta\)-strongness Laver function when \(\theta\) is a limit ordinal of low cofinality,
starting from just a \(\theta\)-strong cardinal. Of course, one can do it starting from just a
little bit more, like a \((\theta+1)\)-strong cardinal, but it is unclear what the sharpest
result is. But, since we are interested in jointness phenomena, let us gloss over this
issue and ask whether the hypotheses in Proposition~\ref{prop:ThetaStrJLDiamond} 
are really necessary in the singular case.

\begin{question}
	Suppose there is a \(\theta\)-strongness Laver function for \(\kappa\) (with \(\theta\)
	possibly being a limit of low cofinality).
	Is there a \(\theta\)-strongness joint Laver sequence of length \(\kappa\)? Or even of length 
	\(\omega\)?
\end{question}

We give a partial answer to this question.
In contrast to the supercompact case, some restrictions are in fact necessary to allow for the
existence of joint Laver sequences for \(\theta\)-strong cardinals. 
The existence of even the
shortest of such sequences can surpass the existence of a \(\theta\)-strong cardinal in
consistency strength.

To give a better lower bound on the consistency strength required, we introduce
a notion of Mitchell rank for \(\theta\)-strong cardinals, inspired by 
Carmody~\cite{Carmody2015:Thesis}.

\begin{definition}
	Let \(\kappa\) be a cardinal and \(\theta\) an ordinal. 
	The \emph{\(\theta\)-strongness Mitchell order} \(\triangleleft\) for \(\kappa\) is defined 
	on the set of \((\kappa,V_\theta)\)-extenders, by letting \(E\triangleleft F\) if
	\(E\) is an element of the (transitive collapse of the) ultrapower of \(V\) by \(F\).
\end{definition}

Unsurprisingly, this Mitchell order has properties analogous to those of the usual
Mitchell order on normal measures on \(\kappa\) or the \(\theta\)-supercompactness Mitchell
order on normal fine measures on
\(\power_\kappa(\theta)\), as studied by Carmody. In particular, the \(\theta\)-strongness
Mitchell order is well-founded and gives rise to a notion of a
\emph{\(\theta\)-strongness Mitchell rank}. Having \(\theta\)-strongness Mitchell
rank at least 2 implies that many cardinals below \(\kappa\) have reflected
versions of \(\theta\)-strongness; for example, if \(\kappa\) has 
\((\kappa+\omega)\)-strongness Mitchell rank at least 2, then there are stationarily
many cardinals \(\lambda<\kappa\) which are \((\lambda+\omega)\)-strong (and much more is true).

We should mention a bound on the \(\theta\)-strongness Mitchel rank of a cardinal \(\kappa\).
If \(j\colon V\to M\) is the ultrapower by a \((\kappa,V_\theta)\)-extender then any
\((\kappa,V_\theta)\)-extenders in \(M\) appear in \(V_{j(\kappa)}^M\). It follows
that these extenders can be written in the from \(j(f)(a)\) for some
function \(f\colon V_\kappa\to V_\kappa\) and
a seed \(a\in V_\theta\). In particular, there are at most \(\beth_\theta\)
many such extenders, counted in \(V\). Any given extender therefore has at most
\(\beth_\theta\) many predecessors in the Mitchell order, so the highest possible
\(\theta\)-strongness Mitchell rank of \(\kappa\) is \(\beth_\theta^+\).

\begin{theorem}
	Let \(\kappa\) be a \(\theta\)-strong cardinal, where \(\theta\) is a limit ordinal,
	and \(\cf(\theta)\leq\kappa<\theta\). 
	If there is a \(\theta\)-strongness joint Laver sequence for \(\kappa\) of length 
	\(\cf(\theta)\)
	then \(\kappa\) has maximal \(\theta\)-strongness Mitchell rank. 
	%In particular,
	%there are stationarily many \(\lambda<\kappa\) which are \((\lambda+\omega)\)-strong.
\end{theorem}

\begin{proof}
	We first show that the existence of a short \(\theta\)-strongness joint Laver sequence 
	implies a certain degree
	of hypermeasurability for \(\kappa\). Let \(\vec{\ell}=
	\langle \ell_\alpha;\alpha<\cf(\theta)\rangle\) be the
	joint Laver sequence. If \(\vec{a}=\langle a_\alpha;\alpha<\cf(\theta)\rangle\) 
	is any sequence of
	targets in \(V_\theta\) there is, by definition, a \(\theta\)-strongness embedding
	\(j\colon V\to M\) with critical point \(\kappa\) such that \(j(\ell_\alpha)(\kappa)=a_\alpha\). 
	But we can recover \(\vec{\ell}\) from \(j(\vec{\ell})\) as an initial segment, since
	\(\vec{\ell}\) is so short. Therefore we actually get the whole sequence \(\vec{a}\in M\),
	just by evaluating that initial segment at \(\kappa\).
	Now consider any \(a\subseteq V_\theta\). We can resolve
	\(a\) into a \(\cf(\theta)\)-sequence of elements \(a_\alpha\) of \(V_\theta\),
	by taking a cofinal sequence \(\seq{\theta_\alpha}{\alpha<\cf(\theta)}\) in \(\theta\)
	and letting \(a_\alpha=a\cap V_{\theta_\alpha}\). It follows that \(a=\bigcup_\alpha a_\alpha\).
	We can take the \(a_\alpha\) as our targets for \(\vec{\ell}\). Our argument then implies that
	there is a \(\theta\)-strongness embedding \(j\colon V\to M\) so that \(\seq{a_\alpha}{\alpha<\cf(\theta)}\in M\), and therefore also \(a\in M\).
	
	Now let \(E\) be an arbitrary \((\kappa,V_\theta)\)-extender.
	Since \(E\) can be represented as a family of measures on \(\kappa\) indexed by \(V_\theta\),
	it is coded by a subset of \(V_\theta\) 
	(using the coding scheme from Lemma~\ref{lemma:FlatCoding}, for example). 
	Applying the argument
	above, there is a \(\theta\)-strongness embedding \(j\colon V\to M\) with critical point
	\(\kappa\) such that \(E\in M\). It follows that \(\kappa\) is \(\theta\)-strong in \(M\),
	giving \(\kappa\) nontrivial \(\theta\)-strongness Mitchell rank in \(V\).
	
	The argument in fact yields more: given any collection of at most
	\(\beth_\theta\) many \((\kappa,V_\theta)\)-extenders, we can code the whole family
	by a subset of \(V_\theta\) and, again, obtain an extender whose ultrapower contains
	the entire collection we started with. It follows that, given any family of at most
	\(\beth_\theta\) many extenders, we can find an extender which is above all of them
	in the \(\theta\)-strongness Mitchell order. Applying this fact inductively now
	shows that \(\kappa\) must have maximal \(\theta\)-strongness Mitchell rank.
\end{proof}

Again, we remind the reader that we have not determined the consistency strength of the existence of a
\(\theta\)-strongness Laver function for \(\theta\) of low cofinality. It may well be that
the high Mitchell rank we just derived from \(\jLdthetastr_{\kappa,\cf(\kappa)}\) can already
be obtained from just \(\Ldthetastr_\kappa\).

%We can conclude that if \(\kappa\) is the least of the class of cardinals \(\lambda\) which
%are \((\lambda+\omega)\)-strong (in the terminology of \cite{Kanamori2005:HigherInfinite} the
%least \(\omega\)-strong cardinal) then there is no \((\kappa+\omega)\)-strongness joint Laver sequence for \(\kappa\) of length \(\omega\).
%The proof easily generalizes to the case where \(\cf(\theta)\leq\kappa\), while requiring
%a larger gap between \(\kappa\) and \(\theta\) yields higher reflected degrees
%of strongness for the \(\lambda\).

Just as in the case of \(\theta\)-supercompactness we can also consider 
\(\Ldthetastr_\kappa\)-trees. In view of Propositions~\ref{prop:ThetaStrJLDiamond}
and \ref{prop:FullyTreeableThetaSCJLDExist} it is not surprising that again,
for most \(\theta\), a single \(\theta\)-strongness Laver diamond yields a
\(\Ldthetastr_\kappa\)-tree.

\begin{proposition}
	\label{prop:ThetaStrLaverTree}
	Suppose \(\kappa\) is \(\theta\)-strong where \(\kappa+2\leq\theta\) and
	\(\theta\) is either a successor ordinal or \(\cf(\theta)>2^\kappa\). Then a
	\(\Ldthetastr_\kappa\)-tree exists if and only if a
	\(\theta\)-strongness Laver function for \(\kappa\) does (if and only if\/
	\(\jLdthetasc_{\kappa,2^\kappa}\) holds).
\end{proposition}

\begin{proof}
	We follow the proof of Proposition~\ref{prop:FullyTreeableThetaSCJLDExist}.
	Note that, since \(\theta\geq\kappa+2\), we have \(\funcs{\kappa}{2}\in V_\theta\),
	so \(V_\theta\) is closed under sequences indexed by \(\funcs{\kappa}{2}\) via the coding
	scheme given by Lemma~\ref{lemma:FlatCoding}.
	If \(\ell\) is a \(\theta\)-strongness Laver function for \(\kappa\) we define
	a \(\Ldthetastr_\kappa\)-tree by letting \(D(t)\) be the element with index \(t\) in the sequence
	coded by \(\ell(|t|)\), if this makes sense. It is now easy to check that this truly is
	a \(\Ldthetastr_\kappa\)-tree: given any sequence of targets \(\vec{a}\) we simply use the Laver function
	\(\ell\) to guess it (or, rather, its code), and the embedding \(j\) obtained this way
	will witness the guessing property for \(D\).
\end{proof}

\section{Joint diamonds}
\label{sec:JD}
Motivated by the joint Laver sequences of the previous sections, we now apply
the jointness concept to smaller cardinals. Of course, since we do not have any 
elementary embeddings of the universe with critical point \(\omega_1\), say,
we need a reformulation that will make sense in this setting as well.

Consider a measurable Laver function \(\ell\) and let \(a\subseteq \kappa\). 
By definition there
is an elementary embedding \(j\colon V\to M\) such that \(j(\ell)(\kappa)=a\). Let \(U\)
be the normal measure on \(\kappa\) derived from this embedding. Since \(U\) is normal, \(a\)
is represented in the ultrapower by the function \(f_a(\xi)=a\cap \xi\), and thus, by
Łoś's theorem, we conclude that \(\ell(\xi)=a\cap\xi\) for \(U\)-almost all \(\xi\).
In particular, the set of such \(\xi\) is stationary in \(\kappa\).
Therefore \(\ell\) is (essentially) a \(\diamond_\kappa\)-sequence. Similarly, if we are 
dealing with a joint Laver sequence \(\langle \ell_\alpha;\alpha<\lambda\rangle\) there is
for every sequence \(\langle a_\alpha;\alpha<\lambda\rangle\) of subsets of \(\kappa\)
a normal measure on \(\kappa\) with respect to which the set
\(\set{\xi<\kappa}{\ell_\alpha(\xi)=a_\alpha\cap \xi}\) has measure one for each \(\alpha\).

This understanding of jointness seems amenable to transfer to smaller cardinals.
There are still no normal measures on \(\omega_1\), but perhaps we can weaken
that requirement slightly.

Recall that a filter on a regular cardinal \(\kappa\) is \emph{normal} if it is closed under
diagonal intersections, and \emph{uniform} if it contains all the cobounded sets. 
It is a standard result that the club filter on \(\kappa\) is the least normal
uniform filter on \(\kappa\) (in fact, a normal filter is uniform iff it extends the
club filter). It follows that any subset of \(\kappa\) contained in a (proper) normal uniform 
filter is stationary. Conversely, if \(S\subseteq\kappa\) is stationary, then it is
easy to check that \(S\), together with the club filter, generates a normal uniform
filter on \(\kappa\). Altogether, we see that a set is stationary iff it is contained
in a normal uniform filter. This observation suggests an analogy between Laver functions
and \(\diamond_\kappa\)-sequences: in the same way that Laver functions guess their
targets on large sets with respect to some large cardinal measure, 
\(\diamond_\kappa\)-sequences guess their targets on large sets with respect to some normal uniform
filter. Extending the analogy, in the same way that a joint Laver sequence is
a collection of Laver functions that guess sequences of targets on large sets
with respect to a \emph{common} large cardinal measure (corresponding to the single
embedding \(j\)), a collection of \(\diamond_\kappa\)-sequences will be joint
if they guess sequences of targets on large sets with respect to a common
normal uniform filter.

We will adopt the following terminology: if \(\kappa\) is a cardinal then a
\emph{\(\kappa\)-list} is a function \(d\colon \kappa\to\mathcal{P}(\kappa)\) with
\(d(\alpha)\subseteq\alpha\) (this term seems to have originated in Wei\ss' dissertation~\cite{Weiss2010:Thesis}).

\begin{definition}
Let \(\kappa\) be an uncountable regular cardinal. A \emph{\(\jd_{\kappa,\lambda}\)-sequence} is a
sequence \(\vec{d}=\langle d_\alpha;\alpha<\lambda\rangle\) of \(\kappa\)-lists 
such that for every sequence \(\langle a_\alpha;\alpha<\lambda\rangle\) of subsets of 
\(\kappa\) there is a (proper) normal uniform filter \(\mathcal{F}\) on \(\kappa\) 
such that for every \(\alpha\) the \emph{guessing set}
\(S_\alpha=S(d_\alpha,a_\alpha)=\set{\xi<\kappa}{d_\alpha(\xi)=a_\alpha\cap \xi}\) 
is in \(\mathcal{F}\).
\end{definition}

An alternative, apparently simpler attempt at defining jointness would be to require
that all the \(\kappa\)-lists in the sequence guess their respective targets on
the same stationary set. Let us say that a \(\jd_{\kappa,\lambda}\)-sequence is
\(\emph{consonant}\) if for any sequence of targets \(\seq{a_\alpha}{\alpha<\lambda}\)
there is a stationary set \(S\) so that \(S\subseteq S_\alpha\) for all \(\alpha<\lambda\).

It is not hard to see that we can derive a consonant \(\jd_{\kappa,\lambda}\)-sequence
from a \(\diamond_\kappa\)-sequence, provided that \(\lambda<\kappa\). We omit the proof, since
the following proposition shows that the consonance requirement is, in the end, too strong to yield a
useful notion of jointness for longer sequences.

\begin{proposition}
	Let \(\kappa\) be an uncountable regular cardinal. There are no consonant \(\jd_{\kappa,\kappa}\)-sequences.
\end{proposition}

\begin{proof}
	Suppose that \(\vec{d}\) is a \(\jd_{\kappa,\kappa}\)-sequence and consider
	the sequence of targets \(\seq{\kappa\setminus d_\alpha(\alpha)}{\alpha<\kappa}\). If \(\vec{d}\)
	were consonant, there would definitely need to be some \(\xi<\kappa\)
	so that each particular \(\kappa\)-list \(d_\alpha\) guessed its target at \(\xi\).
	But we picked the targets in such a way that \(d_\xi(\xi)\) is not a good guess for
	the \(\xi\)th target. Therefore \(\vec{d}\) is not consonant.
\end{proof}

Upon reflection we thus abandon the consonance requirement and insist only on the jointness property
as originally stated in the definition. 
To be sure, every \(\jLd_{\kappa,\lambda}\)-sequence is also (essentially) a \(\jd_{\kappa,\lambda}\)-sequence. This means that, at least in the presence of sufficient large
cardinals, \(\jd_{\kappa,\lambda}\)-sequences exist.
But, as we will see later in Theorem~\ref{thm:DiamondEquivalentToJD},
\(\jd_{\kappa,\lambda}\)-sequences can exist quite independently of large cardinals.

We will not use the following proposition going forward, but it serves to give
another parallel between \(\diamond\)-sequences and Laver diamonds. It turns out that
\(\diamond\)-sequences can be seen as Laver functions, except that they work with
generic elementary embeddings.

\begin{proposition}
\label{prop:JDisGenericLaver}
Let \(\kappa\) be an uncountable regular cardinal and \(d\) a \(\kappa\)-list.
Then \(d\) is a \(\diamond_\kappa\)-sequence iff there is, for any \(a\subset\kappa\),
a generic elementary embedding \(j\colon V\to M\) with critical point \(\kappa\)
and \(M\) wellfounded up to \(\kappa+1\) such that \(j(d)(\kappa)=a\).
\end{proposition}

\begin{proof}
Suppose \(d\) is a \(\diamond_\kappa\)-sequence and fix a target \(a\subseteq\kappa\).
Let \(S(d,a)=\set{\xi<\kappa}{d(\xi)=a\cap\xi}\) be the guessing set. By our discussion above,
there is a normal uniform filter \(\mathcal{F}\) on \(\kappa\) with \(S\in\mathcal{F}\).
Let \(G\) be a generic ultrafilter extending \(\mathcal{F}\) and \(j\colon V\to M\)
the generic ultrapower by \(G\). Then \(M\) is wellfounded up to \(\kappa^+\) and
\(\kappa=[\mathrm{id}]_G\). Since \(S\in G\), Łoś's theorem now implies that
\(j(d)(\kappa)=a\).

Conversely, fix a target \(a\subseteq\kappa\) and suppose that there is a
generic embedding \(j\) with the above properties. We can replace \(j\) with the
induced normal ultrapower embedding and let \(U\) be the derived ultrafilter in
the extension. Since \(j(d)(\kappa)=a\) it follows that \(S(d,a)\in U\). But since
\(U\) extends the club filter, \(S(d,a)\) must be stationary.
\end{proof}

The same proof will also show that a sequence of \(\kappa\)-lists is joint
if they can guess any sequence of targets via a single generic elementary embedding.

The following key lemma gives a ``bottom up'' criterion deciding when a collection of subsets 
of \(\kappa\) (namely, some guessing sets) is contained in a normal uniform filter.
It is completely analogous to the
finite intersection property characterizing containment in a filter. 

\begin{definition}
Let \(\kappa\) be an uncountable regular cardinal. A family \(\mathcal{A}\subseteq
\mathcal{P}(\kappa)\) has the \emph{diagonal intersection property} if
for any function \(f\colon\kappa\to\mathcal{A}\) the diagonal intersection
\(\diag_{\alpha<\kappa}f(\alpha)\) is stationary.
\end{definition}

\begin{lemma}
\label{lemma:DiagonalIntersectionProperty}
Let \(\kappa\) be uncountable and regular and let \(\mathcal{A}\subseteq\mathcal{P}(\kappa)\).
The family \(\mathcal{A}\) is contained in a normal uniform filter on \(\kappa\) iff
\(\mathcal{A}\) satisfies the diagonal intersection property.
\end{lemma}

\begin{proof}
The forward direction is clear, so let us focus on the converse. Consider the family
of sets
\[
E=\Bigl\{C\cap \diag_{\alpha<\kappa}f(\alpha)\st \text{\(C\subseteq\kappa\) is club and \(f\in\funcs{\kappa}{\mathcal{A}}\)}\Bigr\}\,.
\]
We claim that \(E\) is directed under diagonal intersections: any diagonal intersection of
\(\kappa\) many elements of \(E\) contains another element of \(E\). To see this, take
\(C_\alpha\cap \diag_{\beta<\kappa}f_\alpha(\beta)\in E\) for \(\alpha<\kappa\).
Let \(\langle\cdot,\cdot\rangle\) be a pairing function and define 
\(F\colon \kappa\to\lambda\) by \(F(\langle\alpha,\beta\rangle)=f_\alpha(\beta)\).
A calculation then shows that
\[
\diag_{\alpha<\kappa}(C_\alpha\cap\diag_{\beta<\kappa}f_\alpha(\beta))
=\diag_\alpha C_\alpha\cap\diag_\alpha\diag_\beta f_\alpha(\beta)
\supseteq \bigl( \diag_\alpha C_\alpha\cap D \bigr)\cap \diag_\alpha F(\alpha)\,,
\]
where \(D\) is the club of closure points of the pairing function.

It follows that closing \(E\) under supersets yields a normal uniform filter on \(\kappa\).
By considering constant functions \(f\) we also see that every \(a\in \mathcal{A}\) is in 
this filter.
\end{proof}

Lemma~\ref{lemma:DiagonalIntersectionProperty} will be the crucial tool for verifying
\(\jd_{\kappa,\lambda}\). More specifically, we shall often apply the following corollary.

\begin{corollary}
\label{cor:JDIffShortSubsequenceJD}
A sequence \(\vec{d}=\langle d_\alpha;\alpha<\lambda\rangle\)
is a \(\jd_{\kappa,\lambda}\)-sequence iff every subsequence of length \(\kappa\) is a
\(\jd_{\kappa,\kappa}\)-sequence.
\end{corollary}

\begin{proof}
The forward implication is obvious; let us check the converse. Let
\(\vec{a}=\langle a_\alpha;\alpha<\lambda\rangle\) be a sequence of targets and let
\(S_\alpha\) be the corresponding guessing sets. By 
Lemma~\ref{lemma:DiagonalIntersectionProperty} we need to check that the
family \(\mathcal{S}=\set{S_\alpha}{\alpha<\lambda}\)
satisfies the diagonal intersection property. 
%So fix a function 
%\(f\colon \kappa\to\mathcal{S}\) and let \(r=\set{\alpha}{S_\alpha\in f[\kappa]}\).
%By our assumption \(\vec{d}\rest r\) is a \(\jd_{\kappa,|r|}\)-sequence, so
%\(f[\kappa]\) is contained in a normal uniform filter and, in particular,
%\(\diag_\alpha f(\alpha)\) is stationary.
So fix a function \(f\colon\kappa\to\lambda\) and consider the diagonal intersection
\(\diag_\alpha S_{f(\alpha)}\). We may assume without loss of generality that \(f\) is
injective. By our assumption \(\vec{d}\rest f[\kappa]\) is a \(\jd_{\kappa,\kappa}\)-sequence,
so \(\set{S_{f(\alpha)}}{\alpha<\kappa}\) is contained in a normal uniform filter, which
means that \(\diag_\alpha S_{f(\alpha)}\) is stationary.
\end{proof}

This characterization leads to fundamental differences between joint diamonds and
joint Laver diamonds. While the definition of joint diamonds was inspired by large cardinal
phenomena, the absence of a suitable analogue of the diagonal intersection property in the
large cardinal setting provides for some very surprising results. 

\begin{definition}
Let \(\kappa\) be an uncountable regular cardinal. A \(\diamond_\kappa\)-tree is a
function \(D\colon \funcs{<\kappa}{2}\to \mathcal{P}(\kappa)\) such that for any sequence
\(\langle a_s;s\in\funcs{\kappa}{2}\rangle\) of subsets of \(\kappa\)
there is a (proper) normal uniform filter on \(\kappa\) containing all the guessing sets
\(S_s=S(D,a_s)=\set{\xi<\kappa}{D(s\rest\xi)=a_s\cap\xi}\).
\end{definition}

This definition clearly imitates the definition of \(\Ld_\kappa\)-trees. We also have a
correspondence in the style of Proposition~\ref{prop:JDisGenericLaver}: a
\(\diamond_\kappa\)-tree acts like a \(\Ld_\kappa\)-tree using
generic elementary embeddings.

The following theorem, the main result of this section, shows that, in complete contrast
to our experience with joint Laver diamonds, \(\diamond_\kappa\) already implies all of
its joint versions.

\begin{theorem}
\label{thm:DiamondEquivalentToJD}
Let \(\kappa\) be an uncountable regular cardinal.
The following are equivalent:
\begin{enumerate}
\item \(\diamond_\kappa\)
\item \(\jd_{\kappa,\kappa}\)
\item \(\jd_{\kappa,2^\kappa}\)
\item There exists a \(\diamond_\kappa\)-tree.
\end{enumerate}
\end{theorem}

\begin{proof}
For the implication \((1)\Longrightarrow (2)\), let 
\(d\colon\kappa\to\mathcal{P}(\kappa)\) be a \(\diamond_\kappa\)-sequence and fix
a bijection \(f\colon \kappa\to\kappa\times\kappa\). Define
\[
d_\alpha(\xi)= 
%(f[d(\xi)])_\alpha\cap\alpha =
\set{\eta<\alpha}{ (\alpha,\eta)\in f[d(\xi)]}\,.
\]
We claim that \(\langle d_\alpha;\alpha<\kappa\rangle\)
is a \(\jd_{\kappa,\kappa}\)-sequence.

To see this, take a sequence of targets \(\langle a_\alpha;\alpha<\lambda\rangle\) 
and let \(S_\alpha=\set{\xi<\kappa}{d_\alpha(\xi)=a_\alpha\cap\xi}\) be the guessing sets.
The set
\[
T=\biggl\{
\xi<\kappa \st d(\xi) = f^{-1}\biggl[\bigcup_{\alpha<\kappa}\{\alpha\}\times a_\alpha\biggr]\cap\xi \biggr\}
\]
is stationary in \(\kappa\). Let \(\mathcal{F}\) be the filter generated by the club filter 
on \(\kappa\)
together with \(T\). This is clearly a proper filter. 
To see that it is also normal, consider some typical elements \(C_\alpha\cap T\) of \(\mathcal{F}\),
where \(C_\alpha\subseteq\kappa\) is club for each \(\alpha<\kappa\). Then
\(\diag_{\alpha<\kappa}(C_\alpha\cap T)=(\diag_{\alpha<\kappa} C_\alpha)\cap T\)
is also clearly an element of \(\mathcal{F}\).

We claim that we have \(S_\alpha\in \mathcal{F}\) for all \(\alpha<\kappa\),
so that \(\mathcal{F}\) witnesses the defining property of a 
\(\jd_{\kappa,\kappa}\)-sequence.
Since \(f[\xi]=\xi\times\xi\) for club many 
\(\xi<\kappa\), the set
\[
T'= \biggl\{
\xi<\kappa \st d(\xi)=f^{-1}\biggl[
\bigcup_{\alpha<\kappa}\{\alpha\}\times a_\alpha\biggr]\cap f^{-1}[\xi\times\xi]
\biggr\}
\]
is just the intersection of \(T\) with a club, and is therefore in \(\mathcal{F}\). 
But now observe that
\begin{align*}
T' &= \biggl\{
\xi<\kappa \st d(\xi)=f^{-1}\biggl[
\bigcup_{\alpha<\xi}\{\alpha\}\times(a_\alpha\cap\xi)
\biggr]\biggr\}\\
&= \set{\xi<\kappa}{\forall\alpha<\xi\colon d_\alpha(\xi) = a_\alpha\cap\xi}
= \diag_{\alpha<\kappa}S_\alpha\,.
\end{align*}
We see that \(T'\in \mathcal{F}\) is, modulo a bounded set, contained in each \(S_\alpha\) and can thus
conclude, since \(\mathcal{F}\) is uniform, that \(S_\alpha\in F\) for all \(\alpha<\kappa\).

Instead of proving \((2)\Longrightarrow (3)\), it will be easier to show 
\((2)\Longrightarrow (4)\) directly. Since the implications 
\((4)\Longrightarrow (3)\Longrightarrow (1)\) are obvious, this will finish the proof.

Fix a \(\jd_{\kappa,\kappa}\)-sequence \(\vec{d}\).
We proceed to construct the \(\diamond_\kappa\)-tree \(D\) level by level;
in fact the only meaningful work will take place at limit levels.
At a limit stage \(\gamma<\kappa\) we shall let the first \(\gamma\) many 
\(\diamond_\kappa\)-sequences
anticipate the labels and their positions. 
Concretely, consider the sets \(d_\alpha(\gamma)\) for
\(\alpha<\gamma\). For each \(\alpha\) we interpret \(d_{2\alpha}(\gamma)\) as a node on
the \(\gamma\)-th level of \(\funcs{<\kappa}{2}\) and let \(D(d_{2\alpha}(\gamma))=
d_{2\alpha+1}(\gamma)\), provided that there is no interference between the different
\(\diamond_\kappa\)-sequences. If it should happen that for some \(\alpha\neq \beta\)
we get \(d_{2\alpha}(\gamma)=d_{2\beta}(\gamma)\) but 
\(d_{2\alpha+1}(\gamma)\neq d_{2\beta+1}(\gamma)\)
we scrap the whole level and move on with the construction
higher in the tree.
At the end we extend \(D\) to be defined on the nodes of \(\funcs{<\kappa}{2}\) that
were skipped along the way in any way we like.

We claim that the function \(D\) thus constructed is a \(\diamond_\kappa\)-tree. 
To check this, let us fix a sequence of targets 
\(\vec{a}=\langle a_s;s\in\funcs{\kappa}{2}\rangle\)
and let \(S_s\) be the guessing sets.
By Lemma~\ref{lemma:DiagonalIntersectionProperty} it now suffices to check that
\(\diag_{\alpha<\kappa}S_{s_\alpha}\) is stationary for any sequence of branches 
\(\langle s_\alpha;\alpha<\kappa\rangle\).

For \(\alpha<\kappa\) let 
\(T_{2\alpha}=\set{\xi}{s_\alpha^{-1}[\{1\}]\cap\xi=d_{2\alpha}(\xi)}\)
and \(T_{2\alpha+1}=\set{\xi}{a_{s_\alpha}\cap\xi=d_{2\alpha+1}(\xi)}\).
Since our construction was guided by a \(\jd_{\kappa,\kappa}\)-sequence, 
there is a normal uniform
filter on \(\kappa\) which contains every \(T_\alpha\). In particular,
\(T=\diag_{\alpha<\kappa}T_\alpha\) is stationary. By a simple bootstrapping argument, there
is a club \(C\) of limit ordinals \(\gamma\) such that all \(s_\alpha\rest\gamma\) for 
\(\alpha<\gamma\) are distinct. Let \(\gamma\in C\cap T\). We now have
\(s_\alpha^{-1}[\{1\}]\cap\gamma=d_{2\alpha}(\gamma)\) and
\(a_{s_\alpha}\cap\gamma=d_{2\alpha+1}(\gamma)\) for all \(\alpha<\gamma\).
But this means precisely that the construction of \(D\) goes through at level \(\gamma\)
and that \(\gamma\in\bigcap_{\alpha<\gamma}S_{s_\alpha}\), and it follows that
\(C\cap T\subseteq \diag_{\alpha<\kappa} S_{s_\alpha}\), so
\(\diag_{\alpha<\kappa}S_{s_\alpha}\) is stationary.
\end{proof}

We can again consider the treeability of joint diamond sequences, as we did in
Definition~\ref{def:treeability}. We get the following analogue of 
Corollary~\ref{cor:NontreeableSCJLD}.

\begin{theorem}
\label{thm:NontreeableJDConsistent}
If \(\kappa\) is an uncountable regular cardinal and GCH holds,
then after forcing with \(\Add(\kappa,2^\kappa)\) there is a nontreeable 
\(\jd_{\kappa,2^\kappa}\)-sequence.
\end{theorem}

\begin{proof}
Let \(\P=\Add(\kappa,2^\kappa)\) and \(G\subseteq\P\) generic; 
we refer to the \(\alpha\)-th subset added by \(G\) as \(G_\alpha\). 
We will show that the generic \(G\), seen as a sequence of
\(2^\kappa\) many \(\diamond_\kappa\)-sequences in the usual way, is a nontreeable 
\(\jd_{\kappa,2^\kappa}\)-sequence.

Showing that \(G\) is a \(\jd_{\kappa,2^\kappa}\)-sequence requires only minor modifications to
the usual proof that a Cohen subset of \(\kappa\) codes a \(\diamond_\kappa\)-sequence.
Thus, we view each \(G_\alpha\) as a \(\kappa\)-list. Fix a sequence \(\langle \dot{a}_\alpha;\alpha<2^\kappa
\rangle\) of names for subsets of \(\kappa\), a name \(\dot{f}\) for a function
from \(\kappa\) to \(2^\kappa\) and a name \(\dot{C}\) for a club in \(\kappa\)
as well as a condition \(p\in\P\). We will find a condition \(q\leq p\)
forcing that \(\dot{C}\cap\diag_{\alpha<\kappa}S_{\dot{f}(\alpha)}\) is nonempty, where
\(S_{\dot{f}(\alpha)}\) names the set \(\set{\xi<\kappa}{\dot{a}_{\dot{f}(\alpha)}\cap\xi=
G_\alpha(\xi)}\);
this will show that \(G\) codes a \(\jd_{\kappa,2^\kappa}\)-sequence by 
Lemma~\ref{lemma:DiagonalIntersectionProperty}.

We build the condition \(q\) in \(\omega\) many steps. To start with, let \(p_0=p\) and
let \(\gamma_0\) be an ordinal such that \(\dom(p_0)\subseteq 2^\kappa\times\gamma_0\).
By deciding more and more of the function \(\dot{f}\), the targets \(\dot{a}_{\dot{f}(\alpha)}\),
and the club \(\dot{C}\), we now inductively find ordinals \(\gamma_n<\delta_n<\gamma_{n+1}\), 
sets \(B^\alpha_n\subseteq\gamma_n\),
functions \(f_n\) and a descending sequence of
conditions \(p_n\) satisfying \(\dom(p_n)\subseteq 2^\kappa\times\gamma_n\) and
\(p_{n+1}\forces \delta_n\in\dot{C}\) as well as 
\(p_{n+1}\forces \dot{f}\rest\delta_n=f_n\) and \(p_{n+1}\forces \dot{a}_{f_n(\alpha)}=
B^\alpha_n\) for \(\alpha<\delta_n\).
Let \(\gamma_\omega=\sup_n \gamma_n=\sup_n \delta_n\) and
\(p_\omega=\bigcup_n p_n\) and \(f_\omega=\bigcup_n f_n\) and \(B^\alpha_\omega=
\bigcup_n B_n^\alpha\). 
The construction of these ensures that \(\dom(p_\omega)
\subseteq 2^\kappa\times\gamma_\omega\) and \(p_\omega\) forces that
\(\dot{f}\rest\gamma_\omega=f_\omega\) and \(\dot{a}_{\dot{f}(\alpha)}\cap\gamma_\omega=B^\alpha_\omega\)
for \(\alpha<\gamma_\omega\) as well as \(\gamma_\omega\in\dot{C}\). 
To obtain the final condition \(q\) we now simply extend \(p_\omega\)
by placing the code of \(B^\alpha_\omega\) on top of the \(f(\alpha)\)-th column
for all \(\alpha<\gamma_\omega\). It now follows immediately that \(q\forces\gamma_\omega
\in\dot{C}\cap\diag_{\alpha<\kappa}S_{\dot{f}(\alpha)}\).

It remains to show that the generic \(\jd_{\kappa,2^\kappa}\)-sequence is not treeable,
which follows immediately from Lemma~\ref{lemma:GenericNotTreeable}.
\end{proof}

In the case of Laver diamonds we were able to produce models with quite long
joint Laver sequences but no \(\Ld_\kappa\)-trees simply on consistency strength grounds (see
Theorem~\ref{thm:SeparateLongAndTreeableSCJLD}). In other words, we have models where
there are long joint Laver sequences, but none of them are treeable. The situation seems
different for ordinary diamonds, as Theorem~\ref{thm:DiamondEquivalentToJD} tells us
that treeable joint diamond sequences exist as soon as a single diamond sequence
exists. While Theorem~\ref{thm:NontreeableJDConsistent} shows that it is at least
consistent that there are nontreeable such sequences, we should ask whether this is 
simply always the case.

\begin{question}
Is it consistent for a fixed \(\kappa\), for example \(\kappa=\omega_1\), that every \(\jd_{\kappa,2^\kappa}\)-sequence is
treeable? Is it consistent that all \(\jd_{\kappa,2^\kappa}\)-sequences are treeable for
all \(\kappa\)?
\end{question}

\bibliographystyle{amsplain}
\bibliography{bibbase}
\end{document}